\documentclass{amsart}
\usepackage{amsmath,amssymb,amscd,amsthm,graphicx}

\newtheorem{thm}{Theorem}[section]
\newtheorem{pro}[thm]{Proposition}
\newtheorem{lem}[thm]{Lemma}
\newtheorem{cor}[thm]{Corollary}

\newtheorem{rem}[thm]{Remark}
\newtheorem{problem}[thm]{Problem}

\begin{document}

\title{Orbits of $s$-representations with degenerate Gauss mappings}

\author{Osamu Ikawa}
\address{Department of General Education,
Fukushima National College of Technology,
Iwaki, Fukushima, 970-8034 Japan}
\email{ikawa@fukushima-nct.ac.jp}

\author{Takashi Sakai}
\address{Graduate School of Science,
Osaka City University,
3-3-138 Sugimoto, Sumiyoshi-ku, Osaka-shi, Osaka, 558-8585 Japan}
\email{tsakai@sci.osaka-cu.ac.jp}

\author{Hiroyuki Tasaki}
\address{Graduate School of Pure and Applied Science,
University of Tsukuba,
Tsukuba, Ibaraki, 305-8571 Japan}
\email{tasaki@math.tsukuba.ac.jp}




\begin{abstract}
In this paper we study tangentially degeneracy of the orbits of
$s$-representations in the sphere.
We show that an orbit of an $s$-representation is tangentially degenerate
if and only if it is through a long root,
or a short root of restricted root system of type $G_2$.
Moreover these orbits provide many new examples of tangentially degenerate
submanifolds which satisfy the Ferus equality.
\end{abstract}

\maketitle


\section{Introduction}

A submanifold is called tangentially degenerate if its Gauss mapping is degenerate.
The investigation of tangentially degeneracy of submanifolds has long history.
For example the classification of surfaces in $\mathbf{R}^3$ with degenerate 
Gauss mapping is equivalent to the classification of flat surfaces in $\mathbf{R}^3$.
As a result, that is one of planes, cylinders, cones or tangent developable surfaces.
In this paper we shall investigate the Gauss mapping of a submanifold in the sphere,
that is defined as a mapping to a Grassmannian manifold.
The definition of the Gauss mapping, which here we deal with, will be
given in Section $2$.
Ferus \cite{Ferus} obtained a remarkable result
for tangentially degeneracy of submanifolds in the sphere.
He showed that there exists a number, so-called the Ferus number,
such that if the rank of the Gauss mapping is less than the Ferus number,
then a submanifold must be a totally geodesic sphere.
However, in general it is still unknown whether there exist submanifolds
which satisfy the Ferus equality,
that is,
the equality of the Ferus inequality.
In their papers \cite{Ishikawa, IKM, Miyaoka}, Ishikawa, Kimura and Miyaoka
studied submanifolds with degenerate Gauss mappings in the sphere
via a method of isoparametric hypersurfaces.
They showed that Cartan hypersurfaces and some focal submanifolds
of homogeneous isoparametric hypersurfaces are tangentially degenerate.
Moreover, some of them satisfy the Ferus equality.

A homogeneous isoparametric hypersurface in the sphere is obtained
as an orbit of an $s$-representation of a compact symmetric pair of rank $2$.
Therefore we shall study 
submanifolds with degenerate Gauss mappings via a method of symmetric spaces. 
Our strategy is to investigate the space of relative nullity of the orbits.
In fact, the index of relative nullity is equal to the rank of
tangentially degeneracy.
We will study the second fundamental form of the orbits of
$s$-representations by restricted root systems,
and determine their spaces of relative nullity.
As a result, we will obtain that the space of relative nullity of
the orbits through a long root, or a short root of restricted root system
of type $G_2$, is coincide with the root space of that root.
Hence these orbits are tangentially degenerate.
We note that these orbits are weakly reflective submanifolds
as we showed in the previous paper \cite{IST}.
Moreover, we will show that
the orbits of $s$-representations with degenerate Gauss mapping are exhausted
with above orbits.
Finally we shall observe that these orbits provide
many new examples of tangentially degenerate submanifolds in the sphere
which satisfy the Ferus equality.

\section{Submanifolds with degenerate Gauss mappings}
\setcounter{equation}{0}

Let $f : M \longrightarrow S^n$ be an immersion of an $l$-dimensional manifold $M$
into an $n$-dimensional sphere $S^n$.
The Gauss mapping $\gamma$ of $f$ is defined as a mapping from $M$
to a Grassmannian manifold $G_{l+1}(\mathbf{R}^{n+1})$ of all $(l+1)$-dimensional
subspaces in $\mathbf{R}^{n+1}$ by:
\begin{eqnarray*}
\gamma : M &\longrightarrow& G_{l+1}(\mathbf{R}^{n+1}) \\
x &\longmapsto& \mathbf{R} f(x) \oplus T_{f(x)}(f(M)).
\end{eqnarray*}
We denote by $r$ the maximal rank of the Gauss mapping $\gamma$ of an immersion $f$. 
If the Gauss mapping is degenerate, i.e. $r<l$,
then an immersed submanifold $f(M) \subset S^n$ is said to be
{\it tangentially degenerate} or {\it developable}.
We note that $\gamma$ is constant, i.e. $r=0$,
if and only if $f(M)$ is a part of a totally geodesic sphere.

We denote by $h$ and $A$ the second fundamental form and the shape operator of $f$,
respectively.
Chern and Kuiper \cite{CK} introduced the notion of the {\it index of relative nullity}
at $x \in M$, that is the dimension of the vector space
\begin{eqnarray*}
\mathcal N_x &=& \{ X \in T_x(M) \ |  \ h(X, Y) = 0,\ ^\forall Y \in T_x(M) \} \\
&=& \bigcap_{\xi \in T_x^\perp(M)} \ker(A_\xi).
\end{eqnarray*}
It is easy to show $\ker(d\gamma)_x = \mathcal N_x$,
therefore the index of relative nullity is equal to the degeneracy
of the Gauss mapping at each point.

Let $f : M \longrightarrow S^n$ be an immersion of a compact,
connected manifold $M$ of dimension $l$.
Ferus \cite{Ferus} showed that there exists a number $F(l)$,
which only depends on the dimension $l$ of $M$,
such that the inequality $r < F(l)$ implies $r=0$.
Then $f(M)$ must be an $l$-dimensional great sphere in $S^n$.
Here the number $F(l)$ is called the Ferus number and given by
$$
F(l) = \min \{ k \ | \ A(k)+k \geq l \},
$$
where $A(k)$ is the Adams number, that is the maximal number
of linearly independent vector fields at each point
on the $(k-1)$-dimensional sphere $S^{k-1}$.
Any positive integer $k$ can be written as $(2s + 1)2^t$
by some non-negative integers $s$ and $t$.
We write $t = c + 4d$ by some $0 \leq c \leq 3$ and $0 \leq d$.
In this situation the Adams number $A(k)$ can be calculated by
$$
A(k) = 2^c + 8d -1.
$$

Regarding the Ferus inequality,
Ishikawa, Kimura and Miyaoka posed the following problem:

\begin{problem}[\cite{IKM}]
\begin{enumerate}
\item Is the inequality $r < F(l)$ best possible for the implication $r=0$?
Do there exist tangentially degenerate immersions $M^l \to S^n$ with $r = F(l)$?
\item If the above problem is true,
classify tangentially degenerate immersions $M^l \to S^n$ with $r = F(l)$.
\end{enumerate}
\end{problem}

For these problems, they obtained the following results using isoparametric
hypersurfaces in the sphere.
It is well-known that the number $g$ of distinct principal curvatures
of an isoparametric hypersurface in the sphere
is $1$, $2$, $3$, $4$ or $6$.
A minimal isoparametric hypersurface with $g=3$
is called a Cartan hypersurface.

\begin{thm}[\cite{Ishikawa}]
A homogeneous compact hypersurface in the real projective space $\mathbf{R}P^n$
which is tangentially degenerate is projectively equivalent to
a hyperplane or a Cartan hypersurface.
\end{thm}

\begin{thm}[\cite{Miyaoka}]
When $M$ is a homogeneous isoparametric hypersurface in the sphere with $g=6$,
then both focal submanifolds of $M$ are tangentially degenerate.
Moreover, these submanifolds satisfy the Ferus equality.
\end{thm}

\begin{thm}[\cite{IKM}]
When $M$ is a homogeneous isoparametric hypersurface in the sphere with $g=4$,
then one of focal submanifolds of $M$ is tangentially degenerate,
and another one is not.
Moreover, some of them satisfy the Ferus equality.
\end{thm}

\section{Orbits of $s$-representations}
\setcounter{equation}{0}

A linear isotropy representation of a Riemannian symmetric pair
is called an $s$-representation.
In the following section, we will study orbits of $s$-representations
which are tangentially degenerate.
For this purpose, we shall provide some fundamental notions
of orbits of $s$-representations in this section.

Let $G$ be a compact, connected Lie group
and $K$ a closed subgroup of $G$.
Assume that $\theta$ is an involutive automorphism of $G$
and $G_\theta^0 \subset K \subset G_\theta$, where
$$
G_\theta = \{g \in G \mid \theta(g) = g\}
$$
and $G_\theta^0$ is the identity component of $G_\theta$.
Then $(G, K)$ is a compact symmetric pair with respect to $\theta$.
We denote the Lie algebras of $G$ and $K$ 
by $\mathfrak g$ and $\mathfrak k$, respectively.
The involutive automorphism of $\mathfrak g$ induced from $\theta$
will be also denoted by $\theta$.
Then we have
$$
\mathfrak k = \{X \in \mathfrak g \mid \theta(X) = X\}.
$$
Take an inner product $\langle\; ,\; \rangle$ on $\mathfrak g$
which is invariant under $\theta$ and the adjoint representation of $G$.
Set
$$
\mathfrak m = \{X \in \mathfrak g \mid \theta(X) = - X\},
$$
then we have a canonical orthogonal direct sum decomposition
$$
\mathfrak g = \mathfrak k + \mathfrak m.
$$

Fix a maximal abelian subspace $\mathfrak a$ in $\mathfrak m$
and a maximal abelian subalgebra $\mathfrak t$ in $\mathfrak g$
containing $\mathfrak a$.
For $\alpha \in \mathfrak t$ we set
\begin{equation} \label{eq:root space}
\tilde{\mathfrak g}_\alpha
= \{ X \in \mathfrak g^{\mathbf C} \mid
[H, X] = \sqrt{-1} \langle \alpha, H \rangle X\; (H \in \mathfrak t)\}
\end{equation}
and define the root system $\tilde R$ of $\mathfrak g$ by
\begin{equation} \label{eq:root system}
\tilde R = 
\{\alpha \in \mathfrak t - \{0\} \mid
\tilde{\mathfrak g}_\alpha \ne \{0\}\}.
\end{equation}
For $\lambda \in \mathfrak a$ we set
$$
\mathfrak g_\lambda
= \{X \in \mathfrak g^{\mathbf C} \mid
[H, X] = \sqrt{-1} \langle \lambda, H \rangle X \; (H \in \mathfrak a)\}
$$
and define the restricted root system $R$ of $(\mathfrak g, \mathfrak k)$ by
$$
R = \{ \lambda \in \mathfrak a - \{0\} \mid \mathfrak g_\lambda \ne \{0\}\}.
$$
Set
$$
\tilde R_0 = \tilde R \cap \mathfrak k
$$
and denote the orthogonal projection from $\mathfrak t$ to $\mathfrak a$ by
$H \mapsto \bar H$.
Then we have
$$
R = \{\bar\alpha \mid \alpha \in \tilde R - \tilde R_0\}.
$$
We take a basis of $\mathfrak t$ extended from a basis of $\mathfrak a$
and define the lexicographic orderings $>$ on $\mathfrak a$ and $\mathfrak t$
with respect to these bases.
Then for $H \in \mathfrak t$, $\bar H > 0$ implies $H > 0$.
We denote by $\tilde F$ the set of simple roots of $\tilde R$
with respect to the ordering $>$.
Set
$$
\tilde F_0 = \tilde F \cap \tilde R_0,
$$
then the set of simple roots $F$ of $R$ with respect to the ordering $>$
is given by
$$
F = \{\bar\alpha \mid \alpha \in \tilde F - \tilde F_0\}.
$$
We set
$$
\tilde R_+ = \{\alpha \in \tilde R \mid \alpha > 0\}, \qquad
R_+ = \{\lambda \in R \mid \lambda > 0\}.
$$
Then we have
$$
R_+ = \{\bar\alpha \mid \alpha \in \tilde R_+ - \tilde R_0\}.
$$
We also set
$$
\mathfrak k_0 = \{X \in \mathfrak k \mid [X, H] = 0\; (H \in \mathfrak a)\},
$$
and define
$$
\mathfrak k_\lambda
= \mathfrak k \cap (\mathfrak g_\lambda + \mathfrak g_{-\lambda}), \qquad
\mathfrak m_\lambda
= \mathfrak m \cap (\mathfrak g_\lambda + \mathfrak g_{-\lambda})
$$
for $\lambda \in R_+$.
Under these notations, we have the following lemma.

\begin{lem} \label{lem:basis}
\begin{enumerate}
\item We have orthogonal direct sum decompositions
$$
\mathfrak k = \mathfrak k_0 + \sum_{\lambda \in R_+} \mathfrak k_\lambda,
\qquad
\mathfrak m = \mathfrak a + \sum_{\lambda \in R_+} \mathfrak m_\lambda.
$$
\item If $H \in \mathfrak a$ and $\langle \lambda, H \rangle \neq 0$,
then $\mathrm{ad}(H)$ gives a linear isomorphism
between $\mathfrak m_\lambda$ and $\mathfrak k_\lambda$.
\end{enumerate}
\end{lem}

We define a subset $D$ of $\mathfrak a$ by
$$
D = \bigcup_{\lambda \in R}
\{H \in \mathfrak a \mid \langle \lambda, H \rangle = 0\}.
$$
A connected component of $\mathfrak a - D$ is a Weyl chamber.
We set
$$
C = \{H \in \mathfrak a \mid
\langle \lambda, H \rangle > 0\; (\lambda \in F)\}.
$$
Then $C$ is an open convex subset of $\mathfrak a$ and
the closure of $C$ is given by
$$
\bar C = \{H \in \mathfrak a \mid
\langle \lambda, H \rangle \ge 0\; (\lambda \in F)\}.
$$
For a subset $\Delta \subset F$,
we define
$$
C^\Delta = 
\{H \in \bar C \mid \langle \lambda, H \rangle > 0\; (\lambda \in \Delta),\;
\langle \mu, H \rangle = 0\; (\mu \in F-\Delta)\}.
$$

\begin{lem} \label{lem:docomposition of Weyl}
\begin{enumerate}
\item
For $\Delta_1\subset F$, the decomposition
$$
\overline{C^{\Delta_1}}=\bigcup_{\Delta\subset\Delta_1}C^\Delta
$$
is a disjoint union.
In particular,
$\displaystyle\bar C = \bigcup_{\Delta \subset F} C^\Delta$
is a disjoint union.
\item
For $\Delta_1, \Delta_2 \subset F$, $\Delta_1 \subset \Delta_2$
if and only if $C^{\Delta_1} \subset \overline{C^{\Delta_2}}$.
\end{enumerate}
\end{lem}

For each $\lambda \in F$,
we take $H_\lambda \in \mathfrak a$ such that
$$
\langle H_\lambda, \mu \rangle =
\left\{
\begin{array}{ll}
1 & (\mu = \lambda), \\
0 & (\mu \neq \lambda)
\end{array}
\right.
\quad (\mu \in F).
$$
Then, for $\Delta \subset F$, we have
$$
C^\Delta =
\left\{\left.
\sum_{\lambda \in \Delta}t_\lambda H_\lambda\; 
\right|\;
t_\lambda > 0
\right\}.
$$
We set
$$
R^\Delta = R \cap (F - \Delta)_{\mathbf Z}, \qquad
R^\Delta_+ = R^\Delta \cap R_+.
$$

Under these notations, we have the following lemma.

\begin{lem}[\cite{HKT00}] \label{lem:determination of R_+(H) from H}
Fix a subset $\Delta \subset F$.
For $H \in C^\Delta$ we have the following:
\begin{enumerate}
\item $R^\Delta = \{\lambda \in R\; |\; \langle \lambda, H \rangle = 0\}$,
\item $R^\Delta_+ = \{\lambda \in R_+\; |\; \langle \lambda, H \rangle = 0\}$.
\end{enumerate}
\end{lem}

Now we shall study an orbit $\mathrm{Ad}(K)H$
of the linear isotropy representation of $(G,K)$ through $H\in\mathfrak m$.
We set
$$
Z^H_K = \{k \in K \mid \mathrm{Ad}(k)H = H\}.
$$
Then $Z^H_K$ is a closed subgroup of $K$ and the orbit $\mathrm{Ad}(K)H$ is
diffeomorphic to the coset manifold $K/Z_K^H$. 
The Lie algebra $\mathfrak{z}^H_K$ of $Z^H_K$ is given by
$$
\mathfrak{z}^H_K = \{ X \in \mathfrak{k} \mid [H ,X]=0\}.
$$
An orbit  $\mathrm{Ad}(K)H$ is a  submanifold of the hypersphere $S$
of radius $\| H \|$ in $\mathfrak m$.
From \cite{HKT00}, $\mathrm{Ad}(K)H$ is connected.
Since
$$
\mathfrak m=\bigcup_{k\in K}\mathrm{Ad}(k)\bar C,
$$
without loss of generality we may assume $H\in\bar C$.
Moreover, from Lemma \ref{lem:docomposition of Weyl},
there exists $\Delta \subset F$ such that $H \in C^\Delta$.
From Lemma~\ref{lem:basis} we have the following lemma.

\begin{lem}[\cite{IST}] \label{lem:orbits of s-representation}
For $\Delta \subset F$ and $H \in C^\Delta$,
the tangent space $T_H(\mathrm{Ad}(K)H)$
of the orbit $\mathrm{Ad}(K)H$ at $H$ and
the normal space $T_H^\perp(\mathrm{Ad}(K)H)$ in the hypersphere
can be expressed as
\begin{eqnarray}
T_H(\mathrm{Ad}(K)H)
&=& \sum_{\mu \in R_+ - R_+^\Delta} \mathfrak m_\mu, \label{eq:tangent space}\\
T_H^\perp(\mathrm{Ad}(K)H)
&=& \mathfrak a \cap H^\perp + \sum_{\nu \in R_+^\Delta} \mathfrak m_\nu
= \mathrm{Ad}((Z^H_K)_0) (\mathfrak a \cap H^\perp), \label{eq:normal space}
\end{eqnarray}
where $(Z^H_K)_0$ is the identity component of the stabilizer $Z^H_K$
of $H$ in $K$.
\end{lem}

\section{Orbits of $s$-representations with degenerate Gauss mappings}
\setcounter{equation}{0}

\subsection{Tangentially degenerate orbits}

Let $(G,K)$ be a compact symmetric pair.
We assume that $(G, K)$ is irreducible,
namely $K$ acts irreducibly on $\mathfrak m$.
We consider the orbit $\mathrm{Ad}(K)H$ through $H \in \mathfrak{a}$.
In this section, we study the orbits with degenerate Gauss mappings.
Since the tangentially degeneracy of the orbit is invariant
under scalar multiples on the vector space $\mathfrak m$,
we do not discriminate the difference of the length of a vector $H$.
When $(G, K)$ is of rank $1$, $K$ acts on the sphere in $\mathfrak m$ transitively.
Therefore we only consider a symmetric pair whose rank is greater than or equal to $2$.
The following theorem is the main result of this paper.

\begin{thm} \label{thm:main}
An orbit of an $s$-representation is tangentially degenerate
if and only if it is through a long root
(any root when all roots have the same length),
or a short root of restricted root system of type $G_2$.
Let $\lambda \in R$ be such a root.
Then the tangentially degeneracy of the orbit $\mathrm{Ad}(K)\lambda$
is $\ker(d\gamma)_\lambda = \mathfrak m_\lambda$.
\end{thm}

To prove this theorem,
we show the following proposition first.

\begin{pro} \label{pro:criterion of tangentially degenerate}
If the orbit $\mathrm{Ad}(K)H$ through $H \in \mathfrak a$
is tangentially degenerate,
then $H$ is a constant multiple of a restricted root.
\end{pro}

\begin{proof}
First we note that
$$
A_\xi = \mathrm{Ad}(k)^{-1} A_{\mathrm{Ad}(k)\xi} \mathrm{Ad}(k)
$$
for any $\xi \in \mathfrak a \cap H^\perp$ and $k \in (Z^H_K)_0$.
From this we have
\begin{eqnarray*}
\bigcap_{\xi \in T_H^\perp(\mathrm{Ad}(K)H)} \ker A_\xi
&=& \bigcap_{\xi \in \mathrm{Ad}((Z_K^H)_0)(\mathfrak a \cap H^\perp)}
\ker A_\xi \\
&=& \bigcap_{\xi \in \mathfrak a \cap H^\perp \atop k \in (Z_K^H)_0}
\ker A_{\mathrm{Ad}(k)\xi} \\
&=& \bigcap_{\xi \in \mathfrak a \cap H^\perp \atop k \in (Z_K^H)_0}
\ker (\mathrm{Ad}(k) A_{\xi} \mathrm{Ad}(k)^{-1}) \\
&=& \bigcap_{\xi \in \mathfrak a \cap H^\perp \atop k \in (Z_K^H)_0}
\ker (A_{\xi} \mathrm{Ad}(k)^{-1}) \\
&=& \bigcap_{\xi \in \mathfrak a \cap H^\perp \atop k \in (Z_K^H)_0}
\mathrm{Ad}(k) \ker A_\xi\\
&=&\bigcap_{k \in (Z_K^H)_0}\mathrm{Ad}(k)
\bigcap_{\xi \in \mathfrak a\cap H^\perp}\ker A_\xi.
\end{eqnarray*}

For $\xi \in \mathfrak a \cap H^\perp$ the set of eigenvalues of $A_\xi$
is given by
$$
\left\{ -\frac{\langle \lambda, \xi \rangle}{\langle \lambda, H \rangle} \
\Bigg| \ \lambda \in R_+-R_+^\Delta \right\},
$$
and the eigenspace associated with eigenvalue
$-\langle \lambda, \xi \rangle / \langle \lambda, H \rangle$
is given by
$$
\sum_{-\frac{\langle \mu, \xi \rangle}{\langle \mu, H \rangle}
= -\frac{\langle \lambda, \xi \rangle}{\langle \lambda, H \rangle} }
\mathfrak m_\mu.
$$
See \cite{IST} for details.
The space $\ker A_\xi$ is nothing but the eigenspace associated with $0$-eigenvalue.
Thus
$$
\ker A_\xi
= \sum_{\langle \mu, \xi \rangle = 0} \mathfrak m_\mu.
$$
Therefore we have
$$
\bigcap_{\xi \in \mathfrak{a}\cap H^\perp}\ker A_\xi
= \bigcap_{\xi\in \mathfrak{a}\cap H^\perp}
\sum_{\langle \mu, \xi \rangle = 0} \mathfrak m_\mu 
=\sum_{\mu \mathrel{/\!/} H}\mathfrak m_\mu,
$$
hence
\begin{equation}
\bigcap_{\xi \in T_H^\perp(\mathrm{Ad}(K)H)} \ker A_\xi
= \bigcap_{k \in (Z_K^H)_0}
\mathrm{Ad}(k) \sum_{\mu \mathrel{/\!/} H} \mathfrak m_\mu 
\subset \sum_{\mu \mathrel{/\!/} H} \mathfrak m_\mu. \label{eq:4-1}
\end{equation}
Consequently, if $\mathrm{Ad}(K)H$ is tangentially degenerate,
then $H$ must be a constant multiple of a restricted root.
\end{proof}

From the above proposition, hereafter,
we may consider the orbit through a restricted root,
i.e., we may put $H = \lambda \in R_+$.
We set
$$
\Delta =\{\mu \in F \mid \langle \mu, \lambda \rangle > 0\}.
$$
Then we have $\lambda \in C^\Delta$.
If $2\lambda \notin R_+$, then
$\mathfrak k_0 + \mathfrak k_\lambda$ is a Lie subalgebra of $\mathfrak k$.
We denote  by $K(\lambda)$ the analytic subgroup of $K$
which corresponds to $\mathfrak k_0 + \mathfrak k_\lambda$.

\begin{pro}\label{pro:root and tangentially degeneracy}
If $\lambda \in R_+$ satisfies
\begin{enumerate} 
\item[(a)] $2\lambda \notin R_+$,
\item[(b)] $\lambda + \nu \notin R$ and $\lambda - \nu \notin R$
for all $\nu \in R_+^\Delta$,
\end{enumerate}
then $\mathrm{Ad}(K)\lambda$ is tangentially degenerate.
\end{pro}

\begin{proof}
Since the tangent space of the orbit $\mathrm{Ad}(K)\lambda$ at $\lambda$
is given as in (\ref{eq:tangent space}),
the image of $\lambda$ by the Gauss mapping $\gamma$ is
$$
\gamma(\lambda)
= \mathbf{R} \lambda + \sum_{\mu \in R_+ - R_+^\Delta} \mathfrak m_\mu,
$$
and its orthogonal complement in $\mathfrak m$ is
$$
\gamma(\lambda)^\perp
= \mathfrak a \cap \lambda^\perp
+ \sum_{\nu \in R_+^\Delta} \mathfrak m_\nu.
$$
From a rule of the bracket product of root spaces and the assumption (b),
we have
$$
\left[
\mathfrak k_0,
\mathfrak a \cap \lambda^\perp
+ \sum_{\nu \in R_+^\Delta} \mathfrak m_\nu
\right]
\subset \sum_{\nu \in R_+^\Delta} \mathfrak m_\nu,
\qquad
\left[
\mathfrak k_\lambda,
\mathfrak a \cap \lambda^\perp
+ \sum_{\nu \in R_+^\Delta} \mathfrak m_\nu
\right]
= \{0\}.
$$
Therefore
$$
\left[
\mathfrak k_0 + \mathfrak k_\lambda,
\mathfrak a \cap \lambda^\perp
+ \sum_{\nu \in R_+^\Delta} \mathfrak m_\nu
\right]
\subset
\mathfrak a \cap \lambda^\perp
+ \sum_{\nu \in R_+^\Delta} \mathfrak m_\nu.
$$
This yields
$$
\mathrm{Ad}(K(\lambda))
\left(
\mathfrak a \cap \lambda^\perp
+ \sum_{\nu \in R_+^\Delta} \mathfrak m_\nu
\right)
=
\mathfrak a \cap \lambda^\perp
+ \sum_{\nu \in R_+^\Delta} \mathfrak m_\nu.
$$
Hence
$$
\mathrm{Ad}(K(\lambda))\cdot\gamma(\lambda) = \gamma(\lambda).
$$
Since $\gamma$ is $K$-equivariant, we have
$$
\gamma(\mathrm{Ad}(k)\lambda)
= \mathrm{Ad}(k)\gamma(\lambda)
= \gamma(\lambda)
$$
for any $k \in K(\lambda)$.
This means that $\gamma$ is constant on $\mathrm{Ad}(K(\lambda))\lambda$.
It is clear that $\mathrm{Ad}(K(\lambda))\lambda$ is not a point,
since $T_\lambda(\mathrm{Ad}(K(\lambda))\lambda) = \mathfrak m_\lambda$.
Consequently $\mathrm{Ad}(K)\lambda$ is tangentially degenerate.
\end{proof}

We denote by $\delta \in R_+$ the highest root of $R$.

\begin{lem}[\cite{Wolf}] \label{lem:Wolf}
For $\lambda \in R_+$,
$$
\frac{\langle \lambda, \delta \rangle}{\| \delta \|^2}=
\left\{
\begin{array}{ll}
0   & (\mbox{when } \lambda \perp \delta),\\
1   & (\mbox{when } \lambda = \delta),\\
1/2 & (\mbox{otherwise}).
\end{array}
\right.
$$
When $\langle \lambda, \delta \rangle / \| \delta \|^2 = 0$,
then $\lambda -\delta$ is not a root.
When $\langle \lambda, \delta \rangle / \| \delta \|^2 = 1/2$,
then $\lambda -\delta$ is a root.
\end{lem}

\begin{proof}
Since $\delta$ is the highest root, clearly $\lambda + \delta$ is not a root.
We express $\delta$-series containing $\lambda$ as
$\lambda + n\delta \; (p \leq n \leq 0)$.
Then
$$
-2\frac{\langle \lambda, \delta \rangle}{\| \delta \|^2} = p.
$$
Now we shall show $p=0,-1$ or $-2$.
If we assume that $p \leq -3$, then $\mu = \lambda - 3\delta$ is a root.
Then, from the square norm of $3\delta = \lambda - \mu$, we have
$$
9 \| \delta \|^2 = \| \lambda \|^2 + \| \mu \|^2 - 2\langle \lambda, \mu \rangle .
$$
From $\| \lambda \| \leq \| \delta \|, \| \mu \| \leq \| \delta \|$
and Cauchy's inequality
$$
-\langle \lambda, \mu \rangle \leq \| \lambda \| \| \mu \| \leq \| \delta \|^2,
$$
we have $9\| \delta \|^2 \leq 4\| \delta \|^2$.
This is a contradiction.

In the case of $p=0$, $\lambda$ is perpendicular to $\delta$
and $\lambda - \delta$ is not a root.
In the case of $p=-1$,
$\langle \lambda, \delta \rangle / \| \delta \|^2 = 1/2$
and $\lambda - \delta$ is a root.
When $p=-2$, then $\lambda = \delta$ from Cauchy's inequality.
\end{proof}

From Proposition~\ref{pro:root and tangentially degeneracy}
and Lemma~\ref{lem:Wolf} we have the following corollary.

\begin{cor}
The orbit $\mathrm{Ad}(K)\delta$ through the highest root $\delta$
of $R$ is tangentially degenerate.
\end{cor}

Since a long root is conjugate to the highest root
under the action of the Weyl group,
it satisfies the conditions of Proposition \ref{pro:root and tangentially degeneracy}.
Especially in the case where the lengths of all roots are equal,
all roots satisfy the conditions
of Proposition \ref{pro:root and tangentially degeneracy}.
We determine short roots satisfying the conditions
of Proposition \ref{pro:root and tangentially degeneracy}
in the following proposition.

\begin{pro} \label{pro:4-6}
Short roots satisfying the conditions {\rm (a)} and {\rm (b)} of Proposition 4.3
are only short roots of the restricted root system of type $G_2$.
\end{pro}

\begin{proof}
We will follow the notations of root systems in \cite{Bourbaki}.

In the case of type $B$, the restricted root system is given by
$$
R = \{\pm e_i\; |\; 1 \le i \le p \}
\cup\{\pm e_i \pm e_j\; |\; 1 \le i < j \le p \}.
$$
If we add $\pm e_j$ to a short root $\pm e_i\; (i \neq j)$,
then it becomes a root again.
Thus any short root does not satisfy the condition (b).

In the case of type $C$, the restricted root system is given by
$$
R = \{\pm 2e_i\; |\; 1 \le i \le p \}
\cup\{\pm e_i \pm e_j\; |\; 1 \le i < j \le p \}.
$$
Short roots are $\pm e_i \pm e_j$.
By the action of the Weyl group,
it suffices to consider a short root $e_1 + e_2$.
The set of roots which are perpendicular to $e_1 + e_2$ is
$$
\{\pm(e_1 - e_2)\} \cup \{\pm 2e_i \mid 3 \le i \le p\}
\cup \{\pm e_i \pm e_j \mid 3 \le i < j \le p\}.
$$
Since
$$
(e_1 + e_2) + (e_1 - e_2) = 2e_1 \in R,\quad 
(e_1 + e_2) - (e_1 - e_2) = 2e_2  \in R,
$$
$e_1 + e_2$ does not satisfy the condition (b).

In the case of type $G_2$,
we can easily see that 
all short roots satisfy the conditions (a) and (b).

In the case of type $BC$, the restricted root system is given by
$$
R = \{\pm 2e_i, \pm e_i\; |\; 1 \le i \le p\}
\cup\{\pm e_i \pm e_j\; |\; 1 \le i < j \le p\}.
$$
We can see that short roots $\pm e_i,\, \pm e_i \pm e_j$
do not satisfy the condition (b)
by a similar way in the case of types $B$ and $C$.

The root system of $F_4$ contains a root system of type $B_2$
as a sub-system.
Then a short root of type $F_4$ can be regarded as a short root of type $B_2$.
Thus in this case a short root does not satisfy the condition (b).
\end{proof}

By the above discussion we obtained that the orbits stated in 
Theorem~\ref{thm:main} are tangentially degenerate.
In order to determine the spaces of relative nullity of these orbits
and to show other orbits are not tangentially degenerate,
we give the following criterion
for an orbit of an $s$-representation to be tangentially degenerate.

\begin{pro} \label{pro:tangentially degenerate}
The orbit $\mathrm{Ad}(K)\lambda$ through a restricted root $\lambda \in R$
is tangentially degenerate if and only if
there exists a non-zero subspace of
$\sum_{\mu \mathrel{/\!/} \lambda} \mathfrak{m}_\mu$
which is invariant under $\mathrm{ad}(\mathfrak{z}^\lambda_K )$.
More precisely,
\begin{equation} \label{eq:tangentially degenerate}
\ker (d\gamma)_\lambda
=\bigcap_{k\in (Z^\lambda_K)_0}\mathrm{Ad}(k)
\sum_{\mu \mathrel{/\!/} \lambda} \mathfrak{m}_\mu
\end{equation}
and $\ker (d\gamma)_\lambda$ is the maximal subspace of
$\sum_{\mu \mathrel{/\!/} \lambda} \mathfrak{m}_\mu$
which is invariant under $\mathrm{ad}(\mathfrak{z}^\lambda_K )$.
\end{pro}

\begin{proof}
From (\ref{eq:4-1}) we have (\ref{eq:tangentially degenerate}) immediately.
Thus the orbit $\mathrm{Ad}(K)\lambda$ is tangentially degenerate if and only if
the right-hand side of (\ref{eq:tangentially degenerate}) is a non-zero vector space.

If there exists a non-zero subspace $V$ of
$\sum_{\mu \mathrel{/\!/} \lambda}{\mathfrak m}_\mu$
which is invariant under $\mathrm{Ad}((Z^\lambda_K)_0)$,
then 
$$
\bigcap_{k\in (Z^\lambda_K)_0}
\mathrm{Ad}(k)\sum_{\mu \mathrel{/\!/} \lambda}
{\mathfrak m}_\mu
\supset \bigcap_{k\in (Z^\lambda_K)_0}\mathrm{Ad}(k)V=V\not=\{0\}.
$$
Hence $\mathrm{Ad}(K)\lambda$ is tangentially degenerate.
Conversely, we assume that $\mathrm{Ad}(K)\lambda$ is tangentially degenerate.
Then
$$
\bigcap_{k\in (Z^\lambda_K)_0}
\mathrm{Ad}(k)\sum_{\mu \mathrel{/\!/} \lambda} {\mathfrak m}_\mu
\subset \sum_{\mu \mathrel{/\!/} \lambda} {\mathfrak m}_\mu
$$
is a non-zero subspace, and we denote it by $V$.
Then for any $g\in (Z^\lambda_K)_0$ we have
\begin{eqnarray*}
\mathrm{Ad}(g)V
&=& \mathrm{Ad}(g)\bigcap_{k\in (Z^\lambda_K)_0}\mathrm{Ad}(k)
\sum_{\mu \mathrel{/\!/} \lambda} {\mathfrak m}_\mu \\
&=& \bigcap_{k\in (Z^\lambda_K)_0}\mathrm{Ad}(gk)
\sum_{\mu \mathrel{/\!/} \lambda} {\mathfrak m}_\mu =V.
\end{eqnarray*}
Thus $V$ is invariant under $\mathrm{Ad}((Z^\lambda_K)_0)$.
Consequently,
the orbit $\mathrm{Ad}(K)\lambda$ is tangentially degenerate if and only if
there exists a non-zero subspace of
$\sum_{\mu \mathrel{/\!/} \lambda} \mathfrak{m}_\mu$
invariant under $\mathrm{Ad}((Z^\lambda_K)_0)$.
Since $\mathfrak{z}^\lambda_K$ is the Lie algebra of a connected Lie group
$(Z^\lambda_K)_0$, we obtain the assertion.
\end{proof}

In particular,
for an orbit of the adjoint representation of a compact Lie group
we have the following corollary.

\begin{cor} \label{cor:semisimple Lie group}
An adjoint orbit of a compact, connected semisimple Lie group
through a root $\alpha$ is tangentially degenerate
if and only if there exists a non-zero subspace of
$$
\mathfrak g \cap (\mathfrak g_\alpha \oplus \mathfrak g_{-\alpha})
$$
which is invariant under $\mathrm{ad}(\mathfrak{z}^\alpha_G)$.
\end{cor}

\begin{lem} \label{lem:4-8}
Let $\lambda$ be a root and $V$ a non-zero subspace of $\mathfrak{m}_\lambda$.
Then $V$ is invariant under $\mathrm{ad}(\mathfrak{z}^\lambda_K)$
if and only if $V$ is invariant under $\mathrm{ad}(\mathfrak{k}_0)$ and satisfies
$$
\left[\sum_{\nu \in R_+^\Delta}
\mathfrak{k}_\nu,\ V\right] = \{ 0 \}.
$$
In addition,
if the action of $\mathfrak{k}_0$ on $\mathfrak{m}_\lambda$ is irreducible
then $V=\mathfrak{m}_\lambda$.
\end{lem}

\begin{proof}
Since
$$
\mathfrak{z}^\lambda_K =
\{ X \in \mathfrak k \ | \ [X, \lambda] = 0 \}
= \mathfrak{k}_0 \oplus 
\sum_{\nu \in R_+^\Delta} \mathfrak{k}_\nu,
$$
$V$ is invariant under $\mathrm{ad}(\mathfrak{z}^\lambda_K)$
if and only if $V$ is invariant under $\mathrm{ad}(\mathfrak{k}_0)$ and
$$
\left[\sum_{\nu \in R_+^\Delta}
\mathfrak{k}_\nu,\ V \right] \subset V \subset \mathfrak{m}_\lambda.
$$
On the other hand,
$$
\left[\sum_{\nu \in R_+^\Delta}
\mathfrak{k}_\nu,\ V \right]
\subset
\left[\sum_{\nu \in R_+^\Delta}
\mathfrak{k}_\nu,\ \mathfrak{m}_\lambda \right]
\subset 
\sum_{\nu \in R_+^\Delta}
(\mathfrak{m}_{\lambda +\nu}\oplus \mathfrak{m}_{\lambda -\nu}).
$$
Hence we have
$$
\left[\sum_{\nu \in R_+^\Delta}
\mathfrak{k}_\nu,\ V \right] \subset 
\left( \mathfrak{m}_\lambda \cap 
\sum_{\nu \in R_+^\Delta}
(\mathfrak{m}_{\lambda +\nu}\oplus \mathfrak{m}_{\lambda -\nu}) \right)
= \{0\}.
$$
\end{proof}

\begin{lem}
The root space $\mathfrak{m}_\delta$ corresponds to
the highest root $\delta$ is a subspace of
$\sum_{\mu \mathrel{/\!/} \delta} \mathfrak{m}_\mu$
invariant under $\mathrm{ad}(\mathfrak{z}_K^\delta )$.
\end{lem}

\begin{proof}
The Lie algebra $\mathfrak{z}_K^\delta$ of $Z^\delta_K$ is given by
$$
\mathfrak{z}_K^\delta = \{ X \in \mathfrak{k} \mid [X,\delta ]=0\}
=\mathfrak{k}_0 \oplus 
\sum_{\langle \nu, \delta \rangle = 0} \mathfrak{k}_\nu.
$$
From Lemma~\ref{lem:Wolf}, we have $\delta \pm \nu \not\in R$
for any $\nu \in R_+$ which is perpendicular to $\delta$.
Hence from Lemma~\ref{lem:4-8},
$\mathfrak{m}_\delta$ is invariant under $\mathrm{ad}(\mathfrak{z}_K^\delta )$.
\end{proof}

From this lemma, we have the following proposition immediately.

\begin{pro} \label{pro:highest root}
Let $(G,K)$ be a compact symmetric pair.
Then the orbit $\mathrm{Ad}(K)\delta$ through the highest root $\delta$ is
tangentially degenerate.
Moreover,
$\ker(d\gamma )_\delta =\mathfrak{m}_\delta$ except the case of type $BC$.
\end{pro}

Similarly we also have the following proposition immediately. 

\begin{pro}
Let $(G,K)$ be a compact symmetric pair with restricted root system of type $G_2$.
Then the orbit through any root $\lambda$ is
tangentially degenerate.
Moreover, $\ker(d\gamma )_\lambda =\mathfrak{m}_\lambda$.
\end{pro}

\begin{pro}
Let $G$ be a compact connected simple Lie group.
An adjoint orbit of $G$ is tangentially degenerate
if and only if it is through a long root, or a short root
in the case of compact simple Lie group $G_2$.
\end{pro}

\begin{proof}
We have already shown that the orbit through a long root,
or a short root of the simple Lie group $G_2$ is tangentially degenerate.
Therefore it suffices to show that, in the case of $G \neq G_2$,
the orbit $\mathrm{Ad}(G) \alpha$ through a short root $\alpha \in R_+$
is not tangentially degenerate.

Assume that $V$ is a subspace of
$\mathfrak{g}\cap (\mathfrak{g}_\alpha\oplus \mathfrak{g}_{-\alpha})$
invariant under $\mathrm{ad}(\mathfrak z^\alpha_G)$.
Then the complexification
$V^{\mathbf C}\subset \mathfrak{g}_\alpha\oplus \mathfrak{g}_{-\alpha}$ of $V$
is a complex vector space which is invariant under
$\mathrm{ad}(\mathfrak z^\alpha_G)$.
We take $v \in V^{\mathbf C}$ and express as
$v=X_\alpha +X_{-\alpha}\; (X_{\pm\alpha}\in\mathfrak{g}_{\pm\alpha})$.
In this case, from Proposition~\ref{pro:4-6},
there exists $\beta \in R_+$ which satisfies
$\langle \beta, \alpha \rangle = 0$ and $\alpha \pm \beta \in R$.
We take a non-zero vector $X_\beta\in\mathfrak{g}_\beta$.
Then
$$
[X_\beta ,v]=[X_\beta ,X_\alpha ]+[X_\beta ,X_{-\alpha}]
\in (\mathfrak{g}_{\beta +\alpha} \oplus \mathfrak{g}_{\beta -\alpha}) \cap V^{\mathbf C}
=\{0\}.
$$
This shows $X_{\pm\alpha}=0$, since
$[\mathfrak{g}_\beta ,\mathfrak{g}_{\pm\alpha}]=\mathfrak{g}_{\beta\pm\alpha}$.
Thus we obtain $V = \{0\}$.
Hence from Corollary \ref{cor:semisimple Lie group},
$\mathrm{Ad}(G)\alpha$ is not tangentially degenerate.
\end{proof}

In Proposition~\ref{pro:highest root} we obtained the spaces of the relative nullity
of the orbit through a highest root except the case of type $BC$.
In the rest of this subsection, we shall study the space of relative nullity
of the orbit through a highest root in the case of the restricted root system
of type $BC_p$.  In this case we can put
\begin{eqnarray*}
&& R = \{ \pm 2e_i \ | \ 1 \leq i \leq p \} \cup
\{ \pm e_i \ | \ 1 \leq i \leq p \} \cup
\{ \pm e_i \pm e_j \ | \ 1 \leq i < j\leq p \}, \\
&& \lambda = 2e_1.
\end{eqnarray*}
We already know that the space of relative nullity $\mathcal N_\lambda$
of $\mathrm{Ad}(K)\lambda$ satisfies
$$
\mathfrak m_{2e_1} \subset \mathcal N_\lambda
\subset \mathfrak m_{2e_1} + \mathfrak m_{e_1}
$$
and invariant under $\mathrm{ad}(\mathfrak z_K^\lambda)$.
Since
\begin{eqnarray*}
R_+^\Delta &=& \{ \mu \in R_+ \ | \ \langle \lambda, \mu \rangle = 0 \} \\
&=& \{ 2e_i \ | \ 2 \leq i \leq p \} \cup
\{ e_i \ | \ 2 \leq i \leq p \} \cup
\{ e_i \pm e_j \ | \ 2 \leq i < j\leq p \},
\end{eqnarray*}
we have
$$
\mathfrak z_K^\lambda = \mathfrak k_0 + \sum_{\mu \in R_+^\Delta} \mathfrak k_\mu
= \mathfrak k_0 + \sum_{2 \leq i \leq p} \mathfrak k_{2e_i}
+ \sum_{2 \leq i \leq p} \mathfrak k_{e_i}
+ \sum_{2 \leq i < j \leq p} \mathfrak k_{e_i \pm e_j}.
$$

First we determine the space of relative nullity of the orbit through a long root
when $(G,K)$ is a Hermitian symmetric pair with restricted root system of type $BC$.
For this purpose, we recall the following two lemmas.

\begin{lem}[\cite{Song} Lemma 2.3] \label{lem:complex structure}
For a Hermitian symmetric space,
the complex structure $J$ translates restricted root spaces as following:
$$
J \mathfrak m_{e_i \pm e_j} = \mathfrak m_{e_i \mp e_j}, \quad
J \mathfrak m_{e_i} = \mathfrak m_{e_i}, \quad
J \mathfrak a = \sum_{i=1}^p \mathfrak m_{2e_i}.
$$
\end{lem}

We denote the Hopf fibration by $\pi : S^{2n+1} \longrightarrow \mathbf CP^n$.

\begin{lem}[\cite{IKM} Lemma 2.2] \label{lem:Hopf fibration}
Let $M \subset \mathbf CP^n$ be a complex submanifold of complex dimension $k$.
Then $\pi^{-1}(M)$ is a submanifold of dimension $2k+1$
with degenerate Gauss mapping of $S^{2n+1}$ .
Moreover, if $M$ is compact and not a complex projective subspace,
then the rank of Gauss mapping is equal to $2k$.
\end{lem}

Now we shall prove the following proposition.

\begin{pro} \label{Hermitian symmetric space}
Assume that $p \geq 2$.
Let $(G,K)$ be a Hermitian symmetric pair with restricted root system of type $BC_p$.
Then the space of relative nullity $\mathcal N_\lambda$
of the orbit through a long root $\lambda \in R$
is given by $\mathcal N_\lambda = \mathfrak m_\lambda$.
\end{pro}

\begin{proof}
Without loss of generality we can put $\lambda = 2e_1$,
and we consider the orbit $\mathrm{Ad}(K)\lambda$ through $\lambda$.
The tangent space of $\mathrm{Ad}(K)\lambda$ at $\lambda$ is given by
$$
T_\lambda(\mathrm{Ad}(K)\lambda) = \sum_{\mu \in R_+ - R_+^\Delta} \mathfrak m_\mu
= \mathfrak m_{2e_1} + \mathfrak m_{e_1}
+ \sum_{2 \leq i \leq p} \mathfrak m_{e_1 \pm e_i}.
$$
We denote by $\pi : S \longrightarrow \mathbf CP^n$
the Hopf fibration from the hypersphere $S$ in $\mathfrak m$
to the complex projective space.
Then the image $\pi(\mathrm{Ad}(K)\lambda)$ of the orbit $\mathrm{Ad}(K)\lambda$
is a submanifold of $\mathbf CP^n$, and its tangent space at $\pi(\lambda)$ is given by
$$
T_{\pi(\lambda)}(\pi(\mathrm{Ad}(K)\lambda))
= \mathfrak m_{e_1} + \sum_{2 \leq i \leq p} \mathfrak m_{e_1 \pm e_i}.
$$
Therefore from Lemma~\ref{lem:complex structure},
$\pi(\mathrm{Ad}(K)\lambda)$ is a complex submanifold of $\mathbf CP^n$.
Obviously $\pi(\mathrm{Ad}(K)\lambda)$ is not a complex projective subspace
when $p \geq 2$.
Thus from Lemma~\ref{lem:Hopf fibration}
the index of the relative nullity of $\mathrm{Ad}(K)\lambda \subset S$
is equal to $1$.
Hence 
$\mathcal N_\lambda = \mathfrak m_{2e_1}$.
\end{proof}

\begin{pro} \label{pro:quaternion Grassmannian manifold}
In the case of
$(G, K) = (Sp(2p+n), Sp(p) \times Sp(p+n))\; (p \geq 2, n \geq 1)$,
the space of relative nullity of the orbit through a long root $\lambda \in R$
is given by $\mathcal N_\lambda = \mathfrak m_\lambda$.
\end{pro}

\begin{proof}
We shall give the restricted root space decomposition of
$(G, K) = (Sp(2p+n), Sp(p) \times Sp(p+n))$.
We express $\mathfrak g$ as
$$
\mathfrak g = \mathfrak{sp}(2p+n)
= \{ X \in M_{2p+n}(\mathbf H) \ | \ ^t \bar{X} + X = 0 \}.
$$
We define an involutive automorphism $\theta$ on $\mathfrak g$ by
$$
\theta : \mathfrak g \longrightarrow \mathfrak g;
X \longmapsto \left[ \begin{array}{cc} I_p & \\ & -I_{p+n} \end{array} \right] X
\left[ \begin{array}{cc} I_p & \\ & -I_{p+n} \end{array} \right],
$$
where $I_r$ denotes the $r \times r$ identity matrix.
Then the eigenspaces $\mathfrak k$ and $\mathfrak m$ of $\theta$
associated to eigenvalues $\pm 1$ are given by
\begin{eqnarray*}
\mathfrak k
&=& \left\{ \left[ \begin{array}{cc} X & \\ & Y \end{array} \right] \ \bigg| \
X \in \mathfrak{sp}(p),\ Y \in \mathfrak{sp}(p+n) \right\}, \\
\mathfrak m
&=& \left\{ \left[ \begin{array}{cc} & X \\ -^t\bar{X} & \end{array} \right] \ \bigg| \
X \in M_{p,p+n}(\mathbf H) \right\}.
\end{eqnarray*}
We take a maximal abelian subspace $\mathfrak a$ of $\mathfrak m$ by
$$
\mathfrak a = \left\{ \left[
\begin{array}{c|c|c} & T & \\ \hline -T & & \\ \hline & & \end{array}
\right] \ \Bigg| \
T = t_1 E_{11}+ \cdots + t_p E_{pp},\ t_i \in \mathbf{R} \right\},
$$
where $E_{ij}$ denotes a matrix whose $(i,j)$ element is $1$ and all other elements
are $0$.
We define $e_i \in \mathfrak a$ by
$$
e_i = \left[
\begin{array}{c|c|c} & E_{ii} & \\ \hline -E_{ii} & & \\ \hline & & \end{array}
\right].
$$
Then the restricted root system of $(\mathfrak g, \mathfrak k)$
is of type $BC_p$.
We note that, when $n=0$, the restricted root system is of type $C_p$.

In the case of type $BC$,
the restricted root spaces $\mathfrak k_{e_i}$ and $\mathfrak m_{e_i}$
which correspond to $e_i$ are given by
\begin{eqnarray*}
\mathfrak m_{e_i}
&=& \left\{ \sum_{j = 1}^{n}(x_j E_{i,2p+j} - \bar x_j E_{2p+j,i}) \
\bigg| \ x_j \in \mathbf H \right\}, \\
\\
\mathfrak k_{e_i}
&=& \left\{ \sum_{j = 1}^{n}(y_j E_{p+i,2p+j} - \bar y_j E_{2p+j,p+i}) \
\bigg| \ y_j \in \mathbf H \right\}.
\end{eqnarray*}

In order to prove the proposition, we will show that
$\mathcal N_\lambda$ does not contain $\mathfrak m_{e_1}$-component.
We take $X \in \mathfrak m_{e_1}$ arbitrarily.
Then
$[\mathfrak k_{e_2}, X] \subset \mathfrak m_{e_1+e_2} + \mathfrak m_{e_1-e_2}$.
Since $\mathcal N_\lambda$ is invariant under $\mathrm{ad}(\mathfrak z_K^\lambda)$,
we have that if $X \in \mathcal N_\lambda$
then $[\mathfrak k_{e_2}, X] \subset \mathcal N_\lambda
\subset \mathfrak m_{2e_1} + \mathfrak m_{e_1}$.
Therefore, if $X \in \mathcal N_\lambda$ then $[\mathfrak k_{e_2}, X] = \{ 0 \}$.
We can express
$X = \sum_{j = 1}^{n}(x_j E_{1,2p+j} - \bar x_j E_{2p+j,1}) \in \mathfrak m_{e_1}$.
Then
$$
[\mathfrak k_{e_2}, X]
= \left\{ \left( \sum_{j = 1}^{n} x_j \bar y_j \right) E_{1,p+2}
- \left( \sum_{j = 1}^{n} y_j \bar x_j \right) E_{p+2,1} \
\bigg| \ y_j \in \mathbf H \right\}.
$$
This yields $X = 0$.
Thus $\mathcal N_\lambda$ does not contain $\mathfrak m_{e_1}$-component.
Hence $\mathcal N_\lambda = \mathfrak m_\lambda$.
\end{proof}

\subsection{Tangentially non-degenerate orbits}

In the above subsection we have proved that all orbits stated
in Theorem~\ref{thm:main} are tangentially degenerate.
In this subsection,
we shall show that other orbits are not tangentially degenerate.

\begin{pro}
Let $(G,K)$ be a Hermitian symmetric pair.
(Then the restricted root system of $(G,K)$ is of type $C$ or $BC$.)
The orbit $\mathrm{Ad}(K)\lambda$
through $\lambda =e_1+e_2$ is not tangentially degenerate.
\end{pro}

\begin{proof}
It is sufficient to prove that if $X \in \mathfrak{m}_{e_1+e_2}$
satisfies $[\mathfrak{k}_{e_1-e_2},X]=\{0\}$,
then $X=0$.
From the assumption,
$$
0=J[\mathfrak{k}_{e_1-e_2},X]=[\mathfrak{k}_{e_1-e_2},JX].
$$
Therefore we have
$$
0=\langle\mathfrak{a},[\mathfrak{k}_{e_1-e_2},JX]\rangle 
=\langle [\mathfrak{a},\mathfrak{k}_{e_1-e_2}],JX\rangle 
=\langle \mathfrak{m}_{e_1-e_2},JX\rangle.
$$
From Lemma~\ref{lem:complex structure} we have $JX\in \mathfrak{m}_{e_1-e_2}$.
This implies $JX=0$, hence $X=0$.
\end{proof}

In the case of $(G,K)=(F_4,SU(2)\cdot Sp(3))$,
$(G,K)$ is a compact symmetric pair which corresponds to a normal real form.
In this case, we shall show that the orbit through a short root
is not tangentially degenerate (Proposition~\ref{pro:normal real form}).

For this purpose, we shall recall some definitions.
A real form $\mathfrak g$ of a semisimple Lie algebra $\mathfrak l$ over $\mathbf C$
is called {\it normal} if in each Cartan decomposition
$\mathfrak g = \mathfrak k + \mathfrak m$ the space $\mathfrak m$ contains
a maximal abelian subalgebra of $\mathfrak g$.
It is known that
there exists a normal real form for each semisimple Lie algebra over $\mathbf C$,
moreover that is unique up to isomorphism (\cite[Ch.~IX, Theorem~5.10]{Helgason})D

A compact symmetric pair $(G, K)$ is called
{\it compact symmetric pair corresponds to a normal real form}
if the dual $(\mathfrak{g}^*,\mathfrak{k})$ of the orthogonal symmetric Lie algebra
$(\mathfrak{g},\mathfrak{k})$ of $(G,K)$ is a normal real form of the complexification
$\mathfrak g^{\mathbf C}$ of $\mathfrak g$.

\begin{pro} \label{pro:normal real form}
Let $(G,K)$ be a compact symmetric pair which corresponds to a normal real form
with a restricted root system of type $B$, $C$, or $F_4$.
Then the orbit through a short root is not tangentially degenerate.
\end{pro}

\begin{proof}
Since $(G,K)$ is a compact symmetric pair which corresponds to a normal real form,
$\mathfrak{k}$ and $\mathfrak{m}$ can be expressed as
$$
\mathfrak{k} = \sum_{\alpha \in R_+} \mathbf{R} F_\alpha, \quad 
\mathfrak{m} = \mathfrak{t} \oplus \sum_{\alpha \in R_+} \mathbf{R} G_\alpha, \quad 
\mathfrak{k}_\alpha = \mathbf{R} F_\alpha, \quad 
\mathfrak{m}_\alpha = \mathbf{R} G_\alpha,
$$
where $F_\alpha = (E_\alpha - E_{-\alpha}) / \sqrt{2}$
and $G_\alpha = \sqrt{-1}(E_\alpha + E_{-\alpha}) / \sqrt{2}$.
Here $E_\alpha \in \mathfrak{g}_\alpha$ satisfies that,
for $\alpha, \beta \in R$, if $\alpha + \beta \in R$ then
$[E_\alpha, E_\beta] = N_{\alpha, \beta} E_{\alpha + \beta}$
and $N_{\alpha, \beta}$ is non-zero real number which satisfies
$N_{\alpha, \beta} = -N_{-\alpha, -\beta}$.

When $\alpha$ is a short root,
as we showed in the proof of Proposition~\ref{pro:4-6},
there exists $\beta \in R_+$ such that $\alpha \perp \beta$ and $\alpha \pm \beta \in R$.
Then we have
$$
[\mathfrak{k}_\beta, \mathfrak{m}_\alpha]
=\mathbf{R}
(N_{\alpha, \beta} G_{\alpha + \beta} - N_{-\alpha, \beta} G_{\alpha - \beta})
\not=\{0\}.
$$
Thus, from Lemma~\ref{lem:4-8},
the orbit $\mathrm{Ad}(K)\alpha$ through $\alpha$ is not tangentially degenerate.
\end{proof}

From Proposition~\ref{pro:normal real form},
in the case of $(G,K)=(F_4,SU(2)\cdot Sp(3))$,
the orbit through a short root is not tangentially degenerate.

\begin{pro}
In the case of $(G,K) = (SO(2p+n), S(O(p)\times O(p+n)))\;(p \geq 2, n \geq 1)$,
the orbit $\mathrm{Ad}(K)\lambda$ through a short root $\lambda$
is not tangentially degenerate.
\end{pro}

\begin{proof}
In this case the restricted root system $R$ of $(G,K)$ is of type $B_p$, that is
$R = \{\pm e_i \mid 1 \leq i \leq p \} \cup \{\pm e_i\pm e_j \mid 1\leq i<j\leq p\}$.
Without loss of generality we can put $\lambda =e_1$.
The action of $\mathfrak{k}_0=\mathfrak{o}(n)$
on $\mathfrak{m}_\lambda =\mathbf{R}^n$ is irreducible,
thus $\mathfrak{m}_\lambda$ is the only non-zero subspace
of $\mathfrak{m}_\lambda$ invariant under $\mathfrak{k}_0$.
Restricted root spaces $\mathfrak{m}_{e_i}, \mathfrak{k}_{e_i}\, (1 \leq i \leq p)$
are given by
\begin{eqnarray*}
&& \mathfrak{m}_{e_i} = \left\{\left.
\left(
\begin{array}{c|c|c}
 & & X \\ \hline
 & &   \\ \hline
-{}^tX & &
\end{array}
\right)\;\right|\;
X=x_1E_{i1}+\cdots +x_nE_{in},\; x_j \in \mathbf{R}
\right\},\\
&& \mathfrak{k}_{e_i} = \left\{\left.
\left(
\begin{array}{c|c|c}
 & &  \\ \hline
 & & -X  \\ \hline
 & {}^tX &
\end{array}
\right)\;\right|\;
X=x_1E_{i1}+\cdots +x_nE_{in},\; x_j \in \mathbf{R}
\right\}.
\end{eqnarray*}
Therefore, when $i\geq 2$, we have that $e_i$ is perpendicular to $e_1$ and
$$
[\mathfrak{k}_{e_i},\mathfrak{m}_{e_1}] = \mathbf{R}
\left(
\begin{array}{c|c|c}
 & -E_{1i} & \\ \hline
E_{i1} & & \\ \hline
 & & 
\end{array}
\right)\subset\mathfrak{m}_{e_1-e_i}\oplus\mathfrak{m}_{e_1+e_i}.
$$
Hence, from Lemma~\ref{lem:4-8},
the orbit $\mathrm{Ad}(K)\lambda$ is not tangentially degenerate.
\end{proof}

\begin{pro}
In the case of $(G,K)=(Sp(2p+n),Sp(p) \times Sp(p+n))\; (p \geq 2, n \geq 0)$,
the orbit $\mathrm{Ad}(K)\lambda$ through a restricted root $\lambda = e_1+e_2$
is not tangentially degenerate.
\end{pro}

\begin{proof}
When $n \geq 1$, the restricted root system of $(G,K)$ is of type $BC_p$.
And when $n = 0$, the restricted root system is of type $C_p$.
However, we shall consider both cases uniformly.
In order to prove the proposition,
it suffices to show that $\{0\}$ is the only subspace of $\mathfrak{m}_{e_1+e_2}$
invariant under $\mathrm{ad}(\mathfrak z_K^\lambda)$.

Let $V$ be a subspace of $\mathfrak m_{e_1+e_2}$
invariant under $\mathrm{ad}(\mathfrak z_K^\lambda)$.
We take $X \in V$ arbitrarily.
Then $[\mathfrak k_{e_1-e_2}, X] \subset V \subset \mathfrak m_{e_1+e_2}$.
On the other hand,
$[\mathfrak k_{e_1-e_2}, X] \subset \mathfrak m_{2e_1} \oplus \mathfrak m_{2e_2}$.
Therefore we have $[\mathfrak k_{e_1-e_2}, X] = \{ 0 \}$.

Under the notation of the proof
of Proposition~\ref{pro:quaternion Grassmannian manifold},
restricted root spaces $\mathfrak m_{e_i+e_j}$ and $\mathfrak k_{e_i-e_j}$
are given by
\begin{eqnarray*}
&& \mathfrak{m}_{e_i+e_j} = \{x(E_{i,p+j}+E_{p+i,j})-\bar{x}(E_{p+j,i}+E_{j,p+i})
\mid x \in \mathbf H \},\\
&& \mathfrak{k}_{e_i-e_j} = \{y(E_{ij}+E_{p+i,p+j})-\bar{y}(E_{ji}+E_{p+j,p+i})
\mid y \in \mathbf H \}.
\end{eqnarray*}
We put 
$X = x(E_{1,p+2}+E_{p+1,2})-\bar{x}(E_{p+2,1}+E_{2,p+1}) \in V$.
Then
$$
[\mathfrak k_{e_1-e_2}, X]
= \{ (x \bar{y} - y \bar{x})(E_{1,p+1} + E_{p+1,1})
+ (\bar{x} y - \bar{y} x)(E_{2,p+2} + E_{p+2,2}) \ | \
y \in \mathbf H \}.
$$
Therefore $x$ must be zero for the right-hand side to be $\{ 0 \}$.
Hence $V = \{ 0 \}$.
Consequently we have that $\{0\}$ is the only subspace of $\mathfrak{m}_{e_1+e_2}$
invariant under $\mathrm{ad}(\mathfrak z_K^\lambda)$.
\end{proof}

Next we shall show when 
$$
(G,K)=(E_6,SU(2)\cdot SU(6)),\quad
(E_7,SU(2)\cdot SO(12)),\quad
(E_8,SU(2)\cdot E_7),
$$
the orbit through a short root $\lambda$ is not tangentially degenerate. 
In these cases, $G/K$ is a compact quaternionic symmetric space 
whose restricted root system is of type $F_4$. 
See Appendix in detail. 

Let $\nu$ be in $R_+$ such that $\nu\perp \lambda$. 
Note that $[\nu, \mathfrak m_\lambda] = \{0\}$.
We take $X \in \mathfrak{m}_\lambda$ arbitrarily.
From Lemma~\ref{lem:4-8},
it is sufficient to prove that $[\mathfrak{k}_\nu ,X]=0$ implies $X=0$. 
Now we assume that $[\mathfrak{k}_\nu ,X]=0$.
Then, from the Jacobi identity and (2) of Lemma~\ref{lem:basis},
we have
$$
0=[\nu ,[\mathfrak{k}_\nu ,X]]
=[[\nu ,\mathfrak{k}_\nu ],X]+[\mathfrak{k}_\nu ,[\nu ,X]]
=[\mathfrak{m}_\nu ,X].
$$
Hence $[\mathfrak{k}_\nu +\mathfrak{m}_\nu ,X]=0$. 
Applying the inverse $\Phi^{-1}$ of the Cayley transform to 
the equality above, we have
$$
\left[\sum_{\alpha\in \tilde{R}, \pi (\Phi (\alpha ))=\nu}
(\mathbf{R}F_\alpha +\mathbf{R}G_\alpha ),\
\Phi^{-1}(X)\right]=0.
$$
Here we used Lemma~\ref{lem:Cayley}.
Since
$$
\left(\sum_{\alpha\in \tilde{R}, \pi (\Phi (\alpha ))=\nu}
(\mathbf{R}F_\alpha +\mathbf{R}G_\alpha )\right)^{\mathbf{C}}
=\sum_{\alpha\in \tilde{R}, \pi (\Phi (\alpha ))=\nu}
(\mathfrak{g}_\alpha +\mathfrak{g}_{-\alpha}),
$$
we have 
$$
\left[\sum_{\alpha\in \tilde{R}, \pi (\Phi (\alpha ))=\nu}
(\mathfrak{g}_\alpha +\mathfrak{g}_{-\alpha}),\
\Phi^{-1}(X)\right]=0.
$$
Using Lemma~\ref{lem:Cayley} again, we have
\begin{eqnarray*}
\Phi^{-1}(X)\in \Phi^{-1}(\mathfrak{m}_\lambda )
\subset \Phi^{-1}(\mathfrak{k}_\lambda +\mathfrak{m}_\lambda )
&=& \sum_{\alpha\in \tilde{R}, \pi (\Phi (\alpha ))=\lambda}
(\mathbf{R}F_\alpha +\mathbf{R}G_\alpha )\\
&\subset& \sum_{\alpha\in \tilde{R}, \pi (\Phi (\alpha ))=\lambda}
(\mathfrak{g}_\alpha +\mathfrak{g}_{-\alpha}).
\end{eqnarray*}
Hence it is sufficient to prove that if
$$
Y\in \sum_{\alpha\in \tilde{R}, \pi (\Phi (\alpha ))=\lambda}
(\mathfrak{g}_\alpha \oplus  \mathfrak{g}_{-\alpha})
$$
satisfies
\begin{equation} \label{eq:4-2}
\left[\sum_{\alpha\in \tilde{R}, \pi (\Phi (\alpha ))=\nu}
(\mathfrak{g}_\alpha \oplus  \mathfrak{g}_{-\alpha}),\ Y \right] = 0,
\end{equation}
then $Y$ must be $0$. 

We shall prove the above claim for each of the three cases.

\begin{pro} \label{pro:(E_6,SU(2)cdot SU(6))}
In the case of $(G,K)=(E_6,SU(2)\cdot SU(6))$, 
the orbit $\mathrm{Ad}(K)\lambda$ through a short root $\lambda$ 
is not tangentially degenerate.
\end{pro}
\begin{proof}
We may put $\lambda =\pi (\Phi (\alpha_1))$. 
Then $\lambda$ is a short root, and
$$
(\pi\Phi )^{-1}(\lambda )
= \{\alpha\in \tilde{R} \mid \pi (\Phi (\alpha ))=\lambda \}
= \{\alpha_1,\alpha_6\}.
$$
We set $\nu =\pi (\Phi (\alpha_3+\alpha_4+\alpha_5+\alpha_6))$.
Then $\nu$ is a short root perpendicular to $\lambda$, and  
$$
(\pi\Phi )^{-1}(\nu )
= \{ \alpha_3+\alpha_4+\alpha_5+\alpha_6,\
\alpha_1+\alpha_3+\alpha_4+\alpha_5\}.
$$

Now we assume that
$$
Y=x_1E_{\alpha_1}+y_1E_{-\alpha_1}+x_2E_{\alpha_6}+y_2E_{-\alpha_6}\in 
\sum_{\alpha\in \tilde{R}, \pi (\Phi (\alpha ))=\lambda}
(\mathfrak{g}_\alpha \oplus  \mathfrak{g}_{-\alpha})
$$
satisfies the condition (\ref{eq:4-2}).
We note that the set of roots of the form 
$\alpha_3+\alpha_4+\alpha_5+\alpha_6\pm \alpha$
where $\alpha \in (\pi\Phi )^{-1}(\lambda )$ is
$$
\{(\alpha_3+\alpha_4+\alpha_5+\alpha_6)+\alpha_1,
(\alpha_3+\alpha_4+\alpha_5+\alpha_6)-\alpha_6\}.
$$
Therefore we have
\begin{eqnarray*}
\lefteqn{[E_{\alpha_3+\alpha_4+\alpha_5+\alpha_6},Y]} \\ 
&=& x_1N_{\alpha_3+\alpha_4+\alpha_5+\alpha_6,\alpha_1}
E_{\alpha_1+\alpha_3+\alpha_4+\alpha_5+\alpha_6}+
y_2N_{\alpha_3+\alpha_4+\alpha_5+\alpha_6,-\alpha_6}
E_{\alpha_3+\alpha_4+\alpha_5}.
\end{eqnarray*}
This shows that the condition $[E_{\alpha_3+\alpha_4+\alpha_5+\alpha_6},Y]=0$ 
yields $x_1=y_2=0$. 
Similarly the condition $[E_{-(\alpha_3+\alpha_4+\alpha_5+\alpha_6)},Y]=0$ 
yields $y_1=x_2=0$.
Hence we obtain $Y=0$.
\end{proof}

The following two propositions can be proved in a similar way 
to the proof of Proposition~\ref{pro:(E_6,SU(2)cdot SU(6))}. 
So we write only the essentials of their proofs.

\begin{pro}
In the case of $(G,K)=(E_7,SU(2)\cdot SO(12))$, 
the orbit $\mathrm{Ad}(K)\lambda$ through a short root $\lambda$ 
is not tangentially degenerate.
\end{pro}
\begin{proof}
We may put $\lambda =\pi (\Phi (\alpha_4))$. 
Then $\lambda$ is a short root, and
$$
(\pi\Phi )^{-1}(\lambda )
=\{\alpha_4,\; \alpha_4+\alpha_5,\; \alpha_2+\alpha_4,\; 
\alpha_2+\alpha_4+\alpha_5\}.
$$
We set $\nu =\pi (\Phi (\alpha_3+\alpha_4))$.
Then $\nu$ is a short root perpendicular to $\lambda$, and
$$
(\pi\Phi )^{-1}(\nu )
=\{\alpha_3+\alpha_4,\; \alpha_3+\alpha_4+\alpha_5,\;
\alpha_2+\alpha_3+\alpha_4,\; \alpha_2+\alpha_3+\alpha_4+\alpha_5\}.
$$

We get the assertion from the following:
The set of roots of the form $\alpha_3+\alpha_4\pm\alpha$ where 
$\alpha \in (\pi\Phi )^{-1}(\lambda )$ is
$$
\{(\alpha_3+\alpha_4)-\alpha_4,
(\alpha_3+\alpha_4)+(\alpha_2+\alpha_4+\alpha_5)\}.
$$
The set of roots of the form
$\alpha_3+\alpha_4+\alpha_5\pm\alpha$ where
$\alpha \in (\pi\Phi )^{-1}(\lambda )-\{\alpha_4,\alpha_2+\alpha_4+\alpha_5\}$ is
$$
\{(\alpha_3+\alpha_4+\alpha_5 )-(\alpha_4+\alpha_5),\; 
(\alpha_3+\alpha_4+\alpha_5 )+(\alpha_2+\alpha_4)\}.
$$
\end{proof}

\begin{pro}
In the case of $(G,K)=(E_8,SU(2)\cdot E_7)$, 
the orbit $\mathrm{Ad}(K)\lambda$ through a short root $\lambda$ 
is not tangentially degenerate.
\end{pro}
\begin{proof}
We may put $\lambda =\pi (\Phi (\alpha_1))$. 
Then $\lambda$ is a short root, and
$$
(\pi\Phi )^{-1}(\lambda )
=\left\{ \begin{array}{l}
\alpha_1,\alpha_1+\alpha_3,\;
\alpha_1+\alpha_3+\alpha_4,\;
\alpha_1+\alpha_2+\alpha_3+\alpha_4, \\
\alpha_1+\alpha_3+\alpha_4+\alpha_5,\;
\alpha_1+\alpha_2+\alpha_3+\alpha_4+\alpha_5, \\
\alpha_1+\alpha_2+\alpha_3+2\alpha_4+\alpha_5,\;
\alpha_1+\alpha_2+2\alpha_3+2\alpha_4+\alpha_5
\end{array} \right\}.
$$
We set
$\nu = \pi (\Phi (
\alpha_1+\alpha_2+\alpha_3+2\alpha_4+2\alpha_5+2\alpha_6+\alpha_7))$.
Then $\nu$ is a short root perpendicular to $\lambda$, and
$$
(\pi\Phi )^{-1}(\nu )
= \left\{ \begin{array}{l}
\alpha_1+\alpha_2+\alpha_3+2\alpha_4+2\alpha_5+2\alpha_6+\alpha_7,\\
\alpha_1+\alpha_2+2\alpha_3+2\alpha_4+2\alpha_5+2\alpha_6+\alpha_7,\\
\alpha_1+\alpha_2+2\alpha_3+3\alpha_4+2\alpha_5+2\alpha_6+\alpha_7,\\
\alpha_1+2\alpha_2+2\alpha_3+3\alpha_4+2\alpha_5+2\alpha_6+\alpha_7,\\
\alpha_1+\alpha_2+2\alpha_3+3\alpha_4+3\alpha_5+2\alpha_6+\alpha_7,\\
\alpha_1+2\alpha_2+2\alpha_3+3\alpha_4+3\alpha_5+2\alpha_6+\alpha_7,\\
\alpha_1+2\alpha_2+2\alpha_3+4\alpha_4+3\alpha_5+2\alpha_6+\alpha_7,\\
\alpha_1+2\alpha_2+3\alpha_3+4\alpha_4+3\alpha_5+2\alpha_6+\alpha_7
\end{array} \right\}.
$$

We get the assertion from the following:
The set of roots of the form 
$\alpha_1+\alpha_2+\alpha_3+2\alpha_4+2\alpha_5+2\alpha_6+\alpha_7\pm\alpha$
where $\alpha \in (\pi\Phi )^{-1}(\lambda )$ is
$$
\left\{ \begin{array}{l}
(\alpha_1+\alpha_2+\alpha_3+2\alpha_4+2\alpha_5+2\alpha_6+\alpha_7)-\alpha_1,\\
(\alpha_1+\alpha_2+\alpha_3+2\alpha_4+2\alpha_5+2\alpha_6+\alpha_7)+
(\alpha_1+\alpha_2+2\alpha_3+2\alpha_4+\alpha_5)
\end{array} \right\}.
$$
The set of roots of the form 
$\alpha_1+\alpha_2+2\alpha_3+2\alpha_4+2\alpha_5+2\alpha_6+\alpha_7\pm\alpha$
where
$$
\alpha \in (\pi\Phi )^{-1}(\lambda )-
\{\alpha_1,\alpha_1+\alpha_2+2\alpha_3+2\alpha_4+\alpha_5\}
$$
is
$$
\left\{ \begin{array}{l}
(\alpha_1+\alpha_2+2\alpha_3+2\alpha_4+2\alpha_5+2\alpha_6+\alpha_7)
-(\alpha_1+\alpha_3),\\
(\alpha_1+\alpha_2+2\alpha_3+2\alpha_4+2\alpha_5+2\alpha_6+\alpha_7)
+(\alpha_1+\alpha_2+\alpha_3+2\alpha_4+\alpha_5)
\end{array} \right\}.
$$
The set of roots of the form 
$\alpha_1+\alpha_2+2\alpha_3+3\alpha_4+2\alpha_5+2\alpha_6+\alpha_7\pm\alpha$
where
$$
\alpha \in (\pi\Phi )^{-1}(\lambda )-
\left\{ \begin{array}{l}
\alpha_1,\
\alpha_1+\alpha_2+2\alpha_3+2\alpha_4+\alpha_5,\\
\alpha_1+\alpha_3,\
\alpha_1+\alpha_2+\alpha_3+2\alpha_4+\alpha_5
\end{array} \right\}
$$
is
$$
\left\{ \begin{array}{l}
(\alpha_1+\alpha_2+2\alpha_3+3\alpha_4+2\alpha_5+2\alpha_6+\alpha_7)
-(\alpha_1+\alpha_3+\alpha_4),\\
(\alpha_1+\alpha_2+2\alpha_3+3\alpha_4+2\alpha_5+2\alpha_6+\alpha_7)
+(\alpha_1+\alpha_2+\alpha_3+\alpha_4+\alpha_5)
\end{array} \right\}.
$$
The set of roots of the form 
$\alpha_1+2\alpha_2+2\alpha_3+3\alpha_4+2\alpha_5+2\alpha_6
+\alpha_7\pm\alpha$ where
$$
\alpha \in (\pi\Phi )^{-1}(\lambda )-
\left\{ \begin{array}{l}
\alpha_1,\
\alpha_1+\alpha_2+2\alpha_3+2\alpha_4+\alpha_5,\\
\alpha_1+\alpha_3,\
\alpha_1+\alpha_2+\alpha_3+2\alpha_4+\alpha_5,\\
\alpha_1+\alpha_3+\alpha_4,\
\alpha_1+\alpha_2+\alpha_3+\alpha_4+\alpha_5
\end{array} \right\}
$$
is
$$
\left\{ \begin{array}{l}
(\alpha_1+2\alpha_2+2\alpha_3+3\alpha_4+2\alpha_5+2\alpha_6+\alpha_7)
-(\alpha_1+\alpha_2+\alpha_3+\alpha_4),\\
(\alpha_1+2\alpha_2+2\alpha_3+3\alpha_4+2\alpha_5+2\alpha_6+\alpha_7)
+(\alpha_1+\alpha_3+\alpha_4+\alpha_5)
\end{array} \right\}.
$$
\end{proof}

\subsection{List of tangentially degeneracy}

At the end of this section, we give the list of symmetric pairs
whose ranks are equal or greater than $2$
such that the orbits of their $s$-representations have degenerate Gauss mappings.
All of them are orbits through long roots except the case of type $G_2$.
In the case of type $G_2$ both of orbits through a long root and a short root
have degenerate Gauss mappings,
and both of them have the same dimension and the same rank of Gauss mapping.
In Table~\ref{table:list of tangentially degenerate orbits},
we denote the dimension of the orbit by $l$ and
the rank of Gauss mapping by $r$.
Then tangentially degeneracy is equal to $l-r$.

\begin{table}[h]
\begin{tabular}{c|c|l|l|l|l|l}
\hline
type & rank & $\mathfrak g$ & $\mathfrak k$ & $l$ & $r$ & $l-r$ \\ \hline
$A$ & $p$ & $\mathfrak{su}(p+1)$ & $\mathfrak{so}(p+1)$ & $2p-1$ & $2p-2$ & $1$ \\
    & $p$ & $\mathfrak{su}(p+1)^2$ & $\mathfrak{su}(p+1)$ &
            $2(2p-1)$ & $2(2p-2)$ & $2$ \\
    & $p$ & $\mathfrak{su}(2(p+1))$ & $\mathfrak{sp}(p+1)$ &
            $4(2p-1)$ & $4(2p-2)$ & $4$ \\
    & $2$ & $\mathfrak e_6$ & $\mathfrak f_4$ & $24$ & $16$ & 8 \\ \hline
$B$ & $p$ & $\mathfrak{so}(2p+1)^2$ & $\mathfrak{so}(2p+1)$ & $8p-10$ & $8p-12$ & $2$ \\
    & $p$ & $\mathfrak{so}(2p+n)$ & $\mathfrak{so}(p) \oplus \mathfrak{so}(p+n)$ &
            $4p+2n-7$ & $4p+2n-8$ & $1$ \\ \hline
$C$ & $p$ & $\mathfrak{sp}(p)$ & $\mathfrak{u}(p)$ & $2p-1$ & $2p-2$ & $1$ \\
    & $p$ & $\mathfrak{sp}(p)^2$ & $\mathfrak{sp}(p)$ & $4p-2$ & $4p-4$ & $2$ \\
    & $p$ & $\mathfrak{sp}(2p)$ & $\mathfrak{sp}(p) \oplus \mathfrak{sp}(p)$ &
            $8p-5$ & $8p-8$ & $3$ \\
    & $p$ & $\mathfrak{su}(2p)$ & $\mathfrak{su}(p) \oplus \mathfrak{su}(p)
            \oplus \mathbf{R}$ & $4p-3$ & $4p-4$ & $1$ \\
    & $p$ & $\mathfrak{so}(4p)$ & $\mathfrak{u}(2p)$ & $8p-7$ & $8p-8$ & $1$ \\
    & $3$ & $\mathfrak e_7$ & $\mathfrak e_6 \oplus \mathbf{R}$ &
            $33$ & $32$ & $1$ \\ \hline
$D$ & $p$ & $\mathfrak{so}(2p)$ & $\mathfrak{so}(p) \oplus \mathfrak{so}(p)$ &
            $4p-7$ & $4p-8$ & 1 \\
    & $p$ & $\mathfrak{so}(2p)^2$ & $\mathfrak{so}(2p)$ &
            $2(4p-7)$ & $2(4p-8)$ & $2$\\ \hline
$E_6$ & $6$ & $\mathfrak e_6$ & $\mathfrak{sp}(4)$ & $21$ & $20$ & $1$ \\
      & $6$ & $\mathfrak e_6 \oplus \mathfrak e_6$ & $\mathfrak e_6$ &
              $42$ & $40$ & $2$ \\ \hline
$E_7$ & $7$ & $\mathfrak e_7$ & $\mathfrak{su}(8)$ & $33$ & $32$ & $1$ \\
      & $7$ & $\mathfrak e_7 \oplus \mathfrak e_7$ & $\mathfrak e_7$ &
              $66$ & $64$ & $2$\\ \hline
$E_8$ & $8$ & $\mathfrak e_8$ & $\mathfrak{so}(16)$ & $57$ & $56$ & $1$ \\
      & $8$ & $\mathfrak e_8 \oplus \mathfrak e_8$ & $\mathfrak e_8$ &
              $114$ & $112$ & $2$ \\ \hline
$F_4$ & $4$ & $\mathfrak f_4$ & $\mathfrak{su}(2) \oplus \mathfrak{sp}(3)$ &
              $15$ & $14$ & $1$ \\
      & $4$ & $\mathfrak f_4 \oplus \mathfrak f_4$ & $\mathfrak f_4$ &
              $30$ & $28$ & $2$ \\
      & $4$ & $\mathfrak e_6$ & $\mathfrak{su}(2) \oplus \mathfrak{su}(6)$ &
              $21$ & $20$ & $1$ \\
      & $4$ & $\mathfrak e_7$ & $\mathfrak{su}(2) \oplus \mathfrak{so}(12)$ &
              $33$ & $32$ & $1$ \\
      & $4$ & $\mathfrak e_8$ & $\mathfrak{su}(2) \oplus \mathfrak e_7$ &
              $57$ & $56$ & $1$ \\ \hline
$G_2$ & $2$ & $\mathfrak g_2$ & $\mathfrak{so}(4)$ &
              $5$ & $4$ & $1$ \\
      & $2$ & $\mathfrak g_2 \oplus \mathfrak g_2$ & $\mathfrak g_2$&
              $10$ & $8$ & $2$ \\ \hline
$BC$  & $p$ & $\mathfrak{su}(2p+n)$ & $\mathfrak{su}(p) \oplus \mathfrak{su}(p+n)
              \oplus \mathbf{R}$ & $4p+2n-3$ & $4p+2n-4$ & $1$ \\
      & $p$ & $\mathfrak{so}(4p+2)$ & $\mathfrak{u}(2p+1)$ & $8p-3$ & $8p-4$ & $1$ \\
      & $p$ & $\mathfrak{sp}(2p+n)$ & $\mathfrak{sp}(p) \oplus \mathfrak{sp}(p+n)$ &
              $8p+4n-5$ & $8p+4n-8$ & $3$ \\
      & $2$ & $\mathfrak e_6$ & $\mathfrak{so}(10) \oplus \mathbf{R}$ &
              $21$ & $20$ & $1$ \\
\hline
\end{tabular}
\caption{}
\label{table:list of tangentially degenerate orbits}
\end{table}

In the list, we can find many orbits which satisfy the equality $r = F(l)$.
In order to observe this we state some properties of the Ferus number.
The definition of the Ferus number immediately implies $F(l)\leq l$.

\begin{lem}\label{lem:Ferus monotone}
$F(l)\leq F(l+1)$. 
\end{lem}
\begin{proof}
The relation
$\{k\mid A(k)+k\geq l+1\}\subset \{k\mid A(k)+k\geq l\}$ implies
$$
F(l+1)=\min\{k\mid A(k)+k\geq l+1\}
\geq \min\{k\mid A(k)+k\geq l\}=F(l).
$$
\end{proof}

\begin{lem}\label{lem:Ferus power}
$F(2^q)=2^q$. 
\end{lem}
\begin{proof}
It is sufficient to show $A(k)+k<2^q$ for $k<2^q$.
We write $k=2^q-(2s+1)2^t$ by some non-negative integers $s$ and $t$,
and $t=c+4d$ by some $0\leq c\leq 3$ and $d \geq 0$.
Then $t<q$ and we get
$$
A(k)=A(2^q-2^t(2s+1))=A(2^t(2^{q-t}-(2s+1)))
=2^c+8d-1.
$$
Thus
$$
A(k)+k=2^q-\{2^{c+4d}(2s+1)-2^c-8d+1\}.
$$
Here
\begin{eqnarray*}
2^{c+4d}(2s+1)-2^c-8d+1&\geq&2^{c+4d}-2^c-8d+1\\
&=&2^c(2^{4d}-1)-8d+1\\
&\geq& 2^{4d}-8d\geq 1
\end{eqnarray*}
Therefore we obtain $A(k)+k<2^q$D
\end{proof}

\begin{pro}\label{pro:Ferus number}
Assume $q\geq 1$ and write $q=c+4d\;(0\leq c\leq 3,d\geq 0)$.
Then
$$
F(2^q + a)=2^q
$$
holds for any $0 \le a \le 2^c+8d-1$.
\end{pro}
\begin{proof}
Since $q\geq 1$, we have $c\geq 1$ or $d\geq 1$.
Thus $A(2^q)=2^c+8d-1\geq 1$. 
This shows $A(2^q)+2^q=2^q+2^c+8d-1$. 
From Lemmas~\ref{lem:Ferus monotone} and \ref{lem:Ferus power}
we get
$$
2^q\geq F(2^q+2^c+8d-1)\geq F(2^q)=2^q.
$$
\end{proof}

The above proposition shows the following equalities:
\begin{eqnarray*}
&& F(2^q+1) = 2^q \quad (q \geq 1), \\
&& F(2^q+2) = 2^q \quad (q \geq 2), \\
&& F(2^q+3) = 2^q \quad (q \geq 2), \\
&& F(2^q+4) = 2^q \quad (q \geq 3).
\end{eqnarray*}

By the use of the above equalities,
we can see many orbits of the $s$-representations which satisfy the
Ferus equality $F(l) = r$ in Table~\ref{table:list of tangentially degenerate orbits}.
For example, the orbits of the $s$-representations of the following symmetric pairs
through a long root satisfy $F(l)=r$: 
\begin{eqnarray*}
&& (\mathfrak{su}(2^{q-1}+2),\ \mathfrak{so}(2^{q-1}+2)) \quad (q \geq 1), \\
&& (\mathfrak{su}(2^{q-2}+2)^2,\ \mathfrak{su}(2^{q-2}+2)) \quad (q \geq 2), \\
&& (\mathfrak{su}(2(2^{q-3}+2)),\ \mathfrak{sp}({2^{q-3}+2})) \quad (q \geq 3), \\
&& (\mathfrak{e}_6,\ \mathfrak{f}_4),\\
&& (\mathfrak{so}(2p+n),\ \mathfrak{so}(p) \oplus \mathfrak{so}(p+n))
\quad (4p+2n-7=2^q+1, p \geq 2, n \geq 1, q \geq 1), \\
&& (\mathfrak{sp}(2^{q-1}+1),\ \mathfrak{u}(2^{q-1}+1)) \quad (q \geq 1), \\
&& (\mathfrak{sp}(2^{q-2}+1)^2,\ \mathfrak{sp}(2^{q-2}+1)) \quad (q \geq 2), \\
&& (\mathfrak{sp}(2(2^{q-3}+1)),\
\mathfrak{sp}(2^{q-3}+1) \oplus \mathfrak{sp}(2^{q-3}+1)) \quad (q \geq 3),\\
&& (\mathfrak{su}(2(2^{q-2}+1)),\
\mathfrak{su}(2^{q-2}+1) \oplus \mathfrak{su}(2^{q-2}+1) \oplus \mathbf{R})
\quad (q \geq 2), \\
&& (\mathfrak{so}(4(2^{q-3}+1)),\ \mathfrak{u}(2(2^{q-3}+1))) \quad (q \geq 3), \\
&& (\mathfrak{e}_7,\ \mathfrak{e}_6 \oplus \mathbf{R}),\\
&& (\mathfrak{so}(2(2^{q-2}+2)),\
\mathfrak{so}(2^{q-2}+2) \oplus \mathfrak{so}(2^{q-2}+2))
\quad (q \geq 2), \\
&& (\mathfrak{so}(2(2^{q-3}+2))^2,\ \mathfrak{so}(2(2^{q-3}+2))) \quad (q \geq 3), \\
&& (\mathfrak{e}_6 \oplus \mathfrak{e}_6,\ \mathfrak{e}_6), \\
&& (\mathfrak{e}_7,\ \mathfrak{su}(8)), \\
&& (\mathfrak{e}_7 \oplus \mathfrak{e}_7,\ \mathfrak{e}_7), \\
&& (\mathfrak{e}_8,\ \mathfrak{so}(16)), \\
&& (\mathfrak{e}_8 \oplus \mathfrak{e}_8,\ \mathfrak{e}_8), \\
&& (\mathfrak{e}_7,\ \mathfrak{su}(2) \oplus \mathfrak{so}(12)), \\
&& (\mathfrak{e}_8,\ \mathfrak{su}(2) \oplus \mathfrak{e}_7), \\
&& (\mathfrak{su}(2p+n),\ \mathfrak{su}(p) \oplus \mathfrak{su}(p+n) \oplus \mathbf{R})
\quad (4p+2n-3=2^q+1, p \geq 2, n \geq 1, q \geq 1), \\
&& (\mathfrak{sp}(2p+n),\ \mathfrak{sp}(p) \oplus \mathfrak{sp}(p+n))
\quad (8p+4n-5=2^q+3, p \geq 2, n \geq 1, q \geq 2).
\end{eqnarray*}
Furthermore the orbits of $s$-representations of symmetric pairs
$$
(\mathfrak{g}_2,\ \mathfrak{so}(4))
\quad \text{and} \quad
(\mathfrak{g}_2 \oplus \mathfrak{g}_2,\ \mathfrak{g}(2))
$$
through a long root or a short root satisfy the Ferus equality
$F(5)=4$ or $F(10)=8$.

\begin{rem}
When $(G, K)$ is of rank $2$,
the results above were studied by Ishikawa, Kimura and Miyaoka
\cite{IKM}.
\end{rem}

\section{Appendix : Quaternionic symmetric spaces}
\setcounter{equation}{0}

A $4n$-dimensional Riemannian manifold is called quaternion-K\"ahler 
if its holonomy group is contained in $Sp(n)\cdot Sp(1)$. 
A quaternion-K\"ahler manifold is called quaternionic symmetric if 
it is a Riemannian symmetric space (\cite[p.~396]{Besse}). 

We will review a construction of a quaternionic symmetric space 
from a compact simple Lie algebra $\mathfrak{g}$ whose rank is 
greater than or equal to $2$ (see \cite{Wolf} in detail). 
Set $G=\mathrm{Int}(\mathfrak{g})$, which is a compact connected semisimple Lie 
group. 
We denote by $\langle\;,\;\rangle$ a biinvariant Riemannian metric on $G$. 
Take a maximal torus $T$ in $G$ and denote its Lie algebra by $\mathfrak{t}$. 
For $\alpha \in \mathfrak{t}$
we set $\tilde{\mathfrak g}_\alpha$ as (\ref{eq:root space}),
and define root system $\tilde R$ by (\ref{eq:root system}).
We have then 
$$
{\mathfrak g}^{\mathbf{C}}={\mathfrak t}^{\mathbf{C}}
+\sum_{\alpha \in \tilde{R}}\tilde{\mathfrak g}_\alpha .
$$
For $\alpha \in \tilde{R}$ we can take $E_\alpha \in \tilde{\mathfrak g}_\alpha$ 
such that 
\begin{eqnarray*}
&& E_\alpha - E_{-\alpha} \in \mathfrak{g}, \quad
\sqrt{-1}(E_\alpha + E_{-\alpha}) \in \mathfrak{g}, \quad
[E_\alpha, E_{-\alpha}] = -\sqrt{-1} \alpha, \\
&& \left\| \frac{1}{\sqrt{2}}(E_\alpha - E_{-\alpha}) \right\|
= \left\| \frac{\sqrt{-1}}{\sqrt{2}}(E_\alpha + E_{-\alpha}) \right\| = 1,
\end{eqnarray*}
and that if we define $N_{\alpha ,\beta}$ by 
$[E_\alpha ,E_\beta ]=N_{\alpha ,\beta}E_{\alpha +\beta}$, then 
$N_{\alpha ,\beta }=-N_{-\alpha ,-\beta}$ where we put 
$N_{\alpha ,\beta}=0$ if $\alpha +\beta\not\in\tilde{R}$. 
Let $\tilde F$ be a fundamental system of $\tilde{R}$ and denote by $\tilde{R}_+$ 
the set of positive roots with respect to $\tilde F$. 
For $\alpha \in \tilde{R}_+$ set
$$
F_\alpha =\frac{1}{\sqrt{2}}(E_\alpha -E_{-\alpha}),\quad 
G_\alpha =\frac{\sqrt{-1}}{\sqrt{2}}(E_\alpha +E_{-\alpha}), 
$$
then we have 
\begin{equation}
\label{eqn:rootspace}
\mathfrak g ={\mathfrak t}+
\sum_{\alpha \in \tilde{R}_+}(\mathbf{R}F_\alpha + \mathbf{R}G_\alpha ),\quad
\|F_\alpha \|=\|G_\alpha \|=1,\quad 
[F_\alpha ,G_\alpha ]=\alpha .
\end{equation}
For each $\alpha \in \tilde{R}_+$,
we define a subalgebra $\mathfrak g (\alpha)$ of $\mathfrak g$ by
$$
\mathfrak g (\alpha )=
\mathbf{R}\alpha + \mathfrak g \cap 
(\tilde{\mathfrak g}_\alpha + \tilde{\mathfrak g}_{-\alpha})
=\mathbf{R}\alpha +\mathbf{R}F_\alpha 
+\mathbf{R}G_\alpha ,
$$
which is isomorphic to $\mathfrak{su}(2)$. 
We denote the highest root by $\delta \in \tilde{R}_+$. 
By Lemma \ref{lem:Wolf},
$$
s=\exp\mathrm{ad}\left(\frac{2\pi}{\|\delta \|^2}\delta\right)
$$
is an involutive automorphism of $\mathfrak g$. 
The fixed points set $\mathfrak k$ of $s$ in $\mathfrak g$ is given by 
\begin{eqnarray*}
\mathfrak k &=& \mathfrak t + \mathbf{R}F_\delta +
\mathbf{R}G_\delta +
\sum_{\alpha \perp \delta}(\mathbf{R}F_\alpha +
\mathbf{R}G_\alpha )\\
&=& \mathfrak g (\delta) + \mathfrak t \cap \delta^\perp
+ \sum_{\alpha\perp\delta}(\mathbf{R}F_\alpha +\mathbf{R}G_\alpha ).
\end{eqnarray*}
The subalgebras $\mathfrak g (\delta)$ and
$\mathfrak t \cap \delta^\perp + \sum_{\alpha\perp\delta}
(\mathbf{R}F_\alpha +\mathbf{R}G_\alpha)$
are ideals of $\mathfrak k$. 
The $(-1)$-eigenspace $\mathfrak m$ of $s$ is given by 
$$
\mathfrak m =\sum_{\alpha\in \tilde{R}^{\mathfrak m}_+}
(\mathbf{R}F_\alpha +\mathbf{R}G_\alpha )
\quad\mbox{where}\quad
\tilde{R}^{\mathfrak m}_+
= \left\{
\alpha \in \tilde{R}_+\;
\left|\;
\frac{\langle\alpha, \delta\rangle}{\|\delta\|^2} = \frac{1}{2}
\right.
\right\}.
$$
Since there exists a subset $\tilde{R}_+(\delta )$ in 
$\tilde{R}^{\mathfrak m}_+$ such that 
\begin{equation}
\label{eqn:quaternion}
\mathfrak m =\sum_{\alpha\in\tilde{R}_+(\delta )}
(\mathbf{R}F_\alpha +\mathbf{R}G_\alpha +
\mathbf{R}F_{\delta -\alpha}
+\mathbf{R}G_{\delta -\alpha}),
\end{equation}
the dimension of $\mathfrak m$ is a multiple of $4$.

We also denote by $s$ the involutive automorphism 
of $G$ induced from $s$. 
Since the fixed point set of $s$ in $G$ is closed and $G$ is compact, 
the identity component $K$ of the fixed points set is also compact. 
The Lie algebra of $K$ coincides with $\mathfrak k$ and $(G,K)$ is a 
compact symmetric pair. 
Hence the coset manifold $G/K$ is a compact Riemannian symmetric space. 
Moreover $G/K$ is a quaternionic symmetric space since (\ref{eqn:quaternion}) 
defines a quaternionic structure. 
Conversely it is known that 
every compact quaternionic symmetric space is obtained in this way. 
We omit its proof. 
See \cite{Wolf} in detail. 

Quaternionic symmetric spaces have a similar property with 
Hermitian symmetric spaces as we shall mention below: 
Two roots $\gamma_1,\gamma_2\in \tilde{R}_+(\delta )$ are 
said to be {\it strongly orthogonal} if $\gamma_1\pm\gamma_2\not\in\tilde{R}$. 

\begin{pro} \label{pro:sos}
Let $G/K$ be a compact quaternionic symmetric space of rank $p$. 
Then there exist $\tilde R_+(\delta)$ which satisfies (\ref{eqn:quaternion})
and a subset $\{ \gamma_i \}_{1 \leq i \leq p}$ of 
$\tilde{R}_+(\delta)$ consisting of strongly orthogonal roots 
such that 
$$
\mathfrak{a} = \sum_{i=1}^p \mathbf{R}F_{\gamma_i}
$$
is a maximal abelian subspace of $\mathfrak{m}$. 
\end{pro}

The proof requires some preparation.

\begin{lem} \label{lem:sos}
If $\alpha, \beta \in \tilde{R}^{\mathfrak m}_+$ and 
$\alpha +\beta\in\tilde{R}$, then $\alpha +\beta =\delta$. 
\end{lem}
\begin{proof} Since $\alpha ,\beta\in \tilde{R}^{\mathfrak m}_+$, we have
$$
\frac{\langle\alpha +\beta ,\delta\rangle}{\|\delta\|^2}=1.
$$
Using Lemma \ref{lem:Wolf}, 
$\alpha +\beta\in\tilde{R}$ implies $\alpha +\beta =\delta$. 
\end{proof}

\begin{cor} \label{cor:sos}
$[\tilde{\mathfrak g}_\alpha, \tilde{\mathfrak g}_\beta]
\subset \tilde{\mathfrak g}_\delta$ 
for $\alpha ,\beta\in \tilde{R}^{\mathfrak m}_+$.
\end{cor}
\begin{proof} 
If $\alpha + \beta \in \tilde{R}$, Lemma~\ref{lem:sos} implies 
$[\tilde{\mathfrak g}_\alpha, \tilde{\mathfrak g}_\beta]
= \tilde{\mathfrak g}_\delta$.
If $\alpha +\beta\not\in\tilde{R}$, then 
$[\tilde{\mathfrak g}_\alpha, \tilde{\mathfrak g}_\beta ]=\{0\}$. 
\end{proof}

If $Q$ is any subset of $\tilde{R}^{\mathfrak m}_+$, let
$$
\mathfrak m_Q
= \sum_{\alpha \in Q} (\tilde{\mathfrak g}_\alpha + \tilde{\mathfrak g}_{-\alpha}).
$$
Remark that $\mathfrak m_{\tilde{R}^{\mathfrak m}_+} 
= \mathfrak m^{\mathbf C}$. 
For the lowest root $\gamma$ in $Q$, put
$$
Q(\gamma)
= \{\beta \in Q - \{\gamma\} \mid \beta \pm \gamma \notin \tilde{R}\}.
$$
Then $\beta \pm \gamma \notin \tilde{R} \cup \{0\}$ for $\beta \in Q(\gamma)$. 

\begin{lem} \label{lem:centralizer}
We denote by $\mathfrak z_{\mathfrak m_Q}(E_\gamma + E_{-\gamma})$ 
the centralizer of $E_\gamma + E_{-\gamma}$ 
in $\mathfrak m_Q$. 
Then 
$$
\mathfrak z_{\mathfrak m_Q}(E_\gamma + E_{-\gamma})
= \mathfrak m_{Q(\gamma)} + \mathbf C(E_\gamma + E_{-\gamma}).
$$
\end{lem}
\begin{proof}
Since $\beta \pm \gamma \notin \tilde{R}\cup \{0\}$ for $\beta \in Q(\gamma)$, 
we have
$$
[\mathfrak m_{Q(\gamma)},\; \tilde{\mathfrak g}_\gamma + \tilde{\mathfrak g}_{-\gamma}]
= \left[
\sum_{\beta \in Q(\gamma)} (\tilde{\mathfrak g}_\beta + \tilde{\mathfrak g}_{-\beta}),\;
\tilde{\mathfrak g}_\gamma + \tilde{\mathfrak g}_{-\gamma}
\right]
= \{0\}.
$$
Since $E_\gamma + E_{-\gamma}
\in \tilde{\mathfrak g}_\gamma + \tilde{\mathfrak g}_{-\gamma}$, 
we get 
$$
[\mathfrak m_{Q(\gamma)}, E_\gamma + E_{-\gamma}]
= \{0\}.
$$
Hence we have 
$$
\mathfrak m_{Q(\gamma)} + \mathbf C(E_\gamma + E_{-\gamma})
\subset \mathfrak z_{\mathfrak m_Q}(E_\gamma + E_{-\gamma}).
$$
Conversely let 
$X$ be in $\mathfrak z_{\mathfrak m_Q}(E_\gamma + E_{-\gamma})$. 
Since $X \in \mathfrak m_Q$, we can express $X$ as 
$$
X = c_\gamma E_\gamma + c_{-\gamma} E_{-\gamma}
+ \sum_{\beta \in Q'} (c_\beta E_\beta + c_{-\beta} E_{-\beta})\quad 
\mbox{where}\quad Q' = Q - \{\gamma\}.
$$
We consider the components of $[X, E_\gamma + E_{-\gamma}] = 0$ 
in the root space decomposition. 
Since the $\mathfrak t^{\mathbf C}$-component is 
$$
c_\gamma [E_\gamma, E_{-\gamma}] + c_{-\gamma} [E_{-\gamma}, E_\gamma]
= (c_\gamma - c_{-\gamma}) [E_\gamma, E_{-\gamma}],
$$
we have $c_\gamma = c_{-\gamma}$, which implies that
$$
X = c_\gamma (E_\gamma + E_{-\gamma})
+ \sum_{\beta \in Q'} (c_\beta E_\beta + c_{-\beta} E_{-\beta}).
$$
Put 
$$
Y = \sum_{\beta \in Q'} (c_\beta E_\beta + c_{-\beta} E_{-\beta}),
$$
then 
$X = c_\gamma (E_\gamma + E_{-\gamma}) + Y$ and 
\begin{eqnarray*}
0 & = &
[X, E_\gamma + E_{-\gamma}]
= [Y, E_\gamma + E_{-\gamma}] \\
& = &
\sum_{\beta \in Q'}
(c_\beta [E_\beta, E_\gamma] + c_\beta [E_\beta, E_{-\gamma}]
+ c_{-\beta} [E_{-\beta}, E_\gamma] + c_{-\beta} [E_{-\beta}, E_{-\gamma}]).
\end{eqnarray*}
Here $[E_\beta, E_\gamma] \in \tilde{\mathfrak g}_\delta$ and 
$[E_{-\beta}, E_{-\gamma}] \in \tilde{\mathfrak g}_{-\delta}$ 
by Corollary~\ref{cor:sos}. 
Since $\beta, \gamma \in \tilde{R}^{\mathfrak m}_+$, we have 
$\langle\beta - \gamma, \delta\rangle = 0$. 
Clearly we get 
$[E_\beta, E_{-\gamma}] \in \tilde{\mathfrak g}_{\beta-\gamma}$ and 
$[E_{-\beta}, E_\gamma] \in \tilde{\mathfrak g}_{-\beta+\gamma}$. 
Since $\gamma$ is the lowest root in $Q$, we have 
$\beta - \gamma > 0$ for $\beta \in Q'$ and $-\beta + \gamma < 0$, 
which implies that 
$\beta - \gamma \ne \delta,\; -\beta + \gamma \ne \delta$. 
Hence, if $\beta - \gamma \in \tilde{R}$, then 
$c_\beta = 0$ and $c_{-\beta} = 0$.
If $\beta + \gamma \in \tilde{R}$, then $\beta = \delta - \gamma$ 
by Lemma~\ref{lem:sos}. 
In this case, $c_\beta = 0$ and $c_{-\beta} = 0$. 
Hence we get
$$
Y = \sum_{\beta \in Q(\gamma)}
(c_\beta E_\beta + c_{-\beta} E_{-\beta})
\in \mathfrak m_{Q(\gamma)}.
$$
Therefore we get the assertion. 
\end{proof}

\begin{proof}[Proof of Proposition \ref{pro:sos}]
We inductively define a sequence of subsets 
$$
\tilde{R}^{\mathfrak m}_+ = Q_1 \supsetneq Q_2 \supsetneq \cdots
\supsetneq Q_s \supsetneq Q_{s+1} = \emptyset
$$
as follows: 
Let $\gamma_i$ be the lowest root in $Q_i$ and set $Q_{i+1} = Q_i(\gamma_i)$. 
Since the cardinal numbers of $\{Q_i\}$ are strictly monotone decreasing, 
the operation is finished at finitely many times. 
Hence we can define $\gamma_1, \dots, \gamma_s \in \tilde{R}^{\mathfrak m}_+$. 
Set 
$$
\tilde{\mathfrak b}
= \sum_{i=1}^s \mathbf C(E_{\gamma_i} + E_{-\gamma_i})
\subset \mathfrak m^{\mathbf C}.
$$
We shall show that $\tilde{\mathfrak b}$ is a maximal abelian 
subspace of $\mathfrak m^{\mathbf C}$. 
By the definition of $\gamma_i$, two distinct roots $\gamma_i$ and 
$\gamma_j$ are strongly orthogonal. 
In particular $\tilde{\mathfrak b}$ is an abelian subspace. 
In order to prove the maximality of $\tilde{\mathfrak b}$, 
set $\mathfrak m_i = \mathfrak m_{Q_i}$ and define a sequence of 
subspaces in $\mathfrak m^{\mathbf C}$ by 
$$
\mathfrak m^{\mathbf C} = \mathfrak m_1
= \mathfrak m_1 + \tilde{\mathfrak b}
\supset \mathfrak m_2 + \tilde{\mathfrak b}
\supset \cdots
\supset \mathfrak m_s + \tilde{\mathfrak b}
\supset \mathfrak m_{s+1} + \tilde{\mathfrak b}
= \tilde{\mathfrak b}.
$$
We shall show that if $X\in \mathfrak m^{\mathbf C}$ satisfies 
$[X,\tilde{\mathfrak b}]=\{0\}$, then $X\in \tilde{\mathfrak b}$. 
In order to prove this, it is sufficient to show that 
if $X \in \mathfrak m_p + \tilde{\mathfrak b}$ then 
$X \in \mathfrak m_{p+1} + \tilde{\mathfrak b}$. 
We can express $X \in \mathfrak m_p +\tilde{\mathfrak b}$ as 
$$
X=Y+Z\qquad (Y\in \mathfrak m_p,\; Z\in \tilde{\mathfrak b}).
$$
Since $[X,\tilde{\mathfrak b}]=\{0\}$, we have 
$$
0=[X,E_{\gamma_p}+E_{-\gamma_p}]=[Y,E_{\gamma_p}+E_{-\gamma_p}],
$$
which implies that 
$Y \in \mathfrak z_{\mathfrak m_p}(E_{\gamma_p}+E_{-\gamma_p})
= \mathfrak m_{p+1} + \mathbf C (E_{\gamma_p}+E_{-\gamma_p})$ 
by Lemma~\ref{lem:centralizer}.
Hence $X=Y+Z$ is in $\mathfrak m_{p+1} + \tilde{\mathfrak b}$.

Since $\gamma_i + \gamma_j \neq \delta$,
we can take a subset $\tilde R_+(\delta)$ which satisfies (\ref{eqn:quaternion})
and contains $\{ \gamma_i \}_{1 \leq i \leq p}$.

\end{proof}

Hence $\mathfrak{m}$ is given by the following:
\begin{eqnarray*}
\mathfrak{m}&=&\mathfrak{a}+ 
\sum_{i=1}^p (\mathbf{R}G_{\gamma_i}+ 
\mathbf{R}F_{\delta -\gamma_i}+
\mathbf{R}G_{\delta -\gamma_i})\\
& &+\sum_{\alpha\in \tilde{R}_+(\delta )-\{\gamma_1,\cdots ,\gamma_p\}}
(\mathbf{R}F_\alpha +\mathbf{R}G_\alpha +
\mathbf{R}F_{\delta -\alpha} +
\mathbf{R}G_{\delta -\alpha} )
\end{eqnarray*}
When the root system of $G$ is not of type $G_2$, then 
$\|\gamma_1\|=\cdots =\|\gamma_p\|$. 
Set
$$
\mathfrak{b}=\mathfrak{t}\cap\{\gamma_1,\cdots ,\gamma_p\}^\perp ,\quad 
\mathfrak{t}'=\mathfrak{a}+\mathfrak{b},
$$
then $\mathfrak{t}'$ is a maximal abelian subalgebra of $\mathfrak{g}$ 
containing $\mathfrak{a}$. 
We define the Cayley transform $\Phi$ by 
$$
\Phi =
\exp\frac{\pi}{2}\mathrm{ad}\left(\sum_{j=1}^p\frac{G_{\gamma_j}}{\|\gamma_j\|}
\right)\in\mathrm{Aut}(\mathfrak{g}),
$$
and set $\lambda_i=\|\gamma_i\|F_{\gamma_i}$, then 
$$
\Phi (\gamma_i)=\lambda_i,\quad \Phi (H)=H \quad (H\in\mathfrak{b}).
$$
Hence the Cayley transform $\Phi$ maps $\mathfrak{t}$ onto $\mathfrak{t}'$. 
We denote by $R$ the restricted root system of $(G,K)$ with respect to 
$\mathfrak{a}$. 
Let 
$\pi :\mathfrak{t}'=\mathfrak{a}+\mathfrak{b}\rightarrow \mathfrak{a}$ 
be the orthogonal projection, then $R=\pi (\Phi (\tilde{R}))$. 
Since
$$
\alpha \equiv
\sum_{i=1}^p\frac{\langle\alpha ,\gamma_i\rangle}{\|\gamma_i\|^2}\gamma_i
\quad \mod \;\mathfrak{b}\quad\mbox{for}\quad \alpha\in\tilde{R},
$$
we have
$$
\Phi (\alpha )\equiv 
\sum_{i=1}^p\frac{\langle\alpha ,\gamma_i\rangle}{\|\gamma_i\|^2}\lambda_i
\quad \mod \;\mathfrak{b},
$$
which implies that
\begin{equation}
\label{eqn:projection}
\pi (\Phi (\alpha ))=
\sum_{i=1}^p\frac{\langle\alpha ,\gamma_i\rangle}{\|\gamma_i\|^2}\lambda_i.
\end{equation}
In particular
$$
\{\lambda_1,\cdots ,\lambda_p\}\subset
R=\left\{\left.
\sum_{i=1}^p\frac{\langle\alpha ,\gamma_i\rangle}{\|\gamma_i\|^2}\lambda_i\;
\right|\;\alpha\in\tilde{R}
\right\}.
$$
The multiplicity $m(\lambda )$ of 
$\lambda =\pi (\Phi (\alpha ))\in\Sigma\; (\alpha\in\tilde{R})$ is 
given by 
$$
m(\lambda )=\#\{\beta\in\tilde{R}\mid 
\langle\alpha ,\gamma_i\rangle =\langle\beta ,\gamma_i\rangle\}.
$$
By (\ref{eqn:projection}), we have
$$
\left\|\pi (\Phi (\alpha ))\right\|^2
=\sum_{i=1}^p\left\langle\alpha ,\frac{\gamma_i}{\|\gamma_i\|}\right\rangle^2
\leq\|\alpha\|^2,
$$
and the equality holds if and only if 
$\alpha\in\mathrm{span}\{\gamma_1,\cdots ,\gamma_p\}$.
Hence $\|\pi (\Phi (\alpha ))\|^2=\|\alpha\|^2$ for any 
$\alpha\in \tilde{R}$ if and only if $p=\mathrm{rank}(G)$. 

\begin{lem}
$\mathbf{R}G_{\gamma_i}\subset \mathfrak{m}_{\lambda_i},\quad 
\mathbf{R}\gamma_i\subset \mathfrak{k}_{\lambda_i}$.
\end{lem}
\begin{proof}
For $H=\sum x_j\lambda_j\in\mathfrak{a}$, we have
\begin{eqnarray*}
[H,G_{\gamma_i}]
&=&\sum x_j[\|\gamma_j\|F_{\gamma_j},G_{\gamma_i}]
=x_i\|\gamma_i\|[F_{\gamma_i},G_{\gamma_i}]\\
&=&x_i\|\gamma_i\|^2\frac{\gamma_i}{\|\gamma_i\|}
=\langle H,\lambda_i\rangle \frac{\gamma_i}{\|\gamma_i\|},\\
\left[H,\frac{\gamma_i}{\|\gamma_i\|}\right]&=&
\sum x_j\|\gamma_j\|\left[F_{\gamma_j},\frac{\gamma_i}{\|\gamma_i\|}\right]
=-x_i\|\gamma_j\|^2G_{\gamma_j}\\
&=&-\langle H,\lambda_i\rangle G_{\gamma_j},
\end{eqnarray*}
where we used (\ref{eqn:rootspace}). 
\end{proof}

\begin{lem}\label{lem:Cayley}
$$
\mathfrak{k}_\lambda +\mathfrak{m}_\lambda 
=\Phi \left(\sum_{\alpha\in\tilde{R},\pi (\Phi (\alpha ))=\lambda}
(\mathbf{R}F_\alpha +\mathbf{R}G_\alpha )\right).
$$
\end{lem}
\begin{proof} 
Since
\begin{eqnarray*}
\mathfrak{k}_\lambda +\mathfrak{m}_\lambda 
&=&\{X\in\mathfrak{g}\mid 
[H,[H,X]]=-\langle\lambda ,H\rangle^2 X\quad (H\in\mathfrak{a})\},\\
\mathbf{R}F_\alpha +\mathbf{R}G_\alpha
&=& \mathfrak{g} \cap (\tilde{\mathfrak g}_\alpha + \tilde{\mathfrak g}_{-\alpha})\\
&=&\{X\in\mathfrak{g}\mid 
[H,[H,X]]=-\langle\alpha ,H\rangle^2 X\quad (H\in\mathfrak{t})\},
\end{eqnarray*}
we have
\begin{eqnarray*}
& &\Phi \left(\sum_{\alpha\in\tilde{R},\pi (\Phi (\alpha ))=\lambda}
(\mathbf{R}F_\alpha +\mathbf{R}G_\alpha )\right)\\
&=&\Phi \left(\sum_{\alpha\in\tilde{R},\pi (\Phi (\alpha ))=\lambda}
\{X\in\mathfrak{g}\mid 
[H,[H,X]]=-\langle\alpha ,H\rangle^2 X\quad (H\in\mathfrak{t})\}\right)\\
&=&\sum_{\alpha\in\tilde{R},\pi (\Phi (\alpha ))=\lambda}
\{Y\in\mathfrak{g}\mid 
[\Phi (H),[\Phi (H),Y]]=-\langle\Phi (\alpha ),\Phi (H)\rangle^2Y 
\quad (H\in\mathfrak{t})\}\\
&=&\sum_{\alpha\in\tilde{R},\pi (\Phi (\alpha ))=\lambda}
\{Y\in\mathfrak{g}\mid 
[H,[H,Y]]=-\langle\Phi (\alpha ),H\rangle^2Y \quad (H\in\mathfrak{t}')\}\\
&\subset&\sum_{\alpha\in\tilde{R},\pi (\Phi (\alpha ))=\lambda}
\{Y\in\mathfrak{g}\mid 
[H,[H,Y]]=-\langle\pi (\Phi (\alpha )),H\rangle^2Y \quad (H\in\mathfrak{a})\}\\
&=&\mathfrak{k}_\lambda +\mathfrak{m}_\lambda .
\end{eqnarray*}
Here $\dim (\mathfrak{k}_\lambda +\mathfrak{m}_\lambda )=2m(\lambda )$. 
Since $\Phi$ is a linear isomorphism, we have
\begin{eqnarray*}
\dim \Phi \left(\sum_{\alpha\in \tilde{R}, \pi (\Phi (\alpha ))=\lambda}
(\mathbf{R}F_\alpha +\mathbf{R}G_\alpha )\right)
&=&\dim\sum_{\alpha\in \tilde{R}, \pi (\Phi (\alpha ))=\lambda}
(\mathbf{R}F_\alpha +\mathbf{R}G_\alpha )\\
&=&2\# \{\alpha\in \tilde{R} \mid \pi (\Phi (\alpha ))=\lambda\}\\
&=&2m(\lambda ).
\end{eqnarray*}
Hence we get the assertion.
\end{proof}

\end{document}